\documentclass[11pt, leqno]{article}

\usepackage{amsmath,amssymb,amsthm, epsfig}

\usepackage{hyperref}

\usepackage{amsmath}
\usepackage{comment}
\usepackage[pdftex]{color}

\usepackage{mathtools}
\usepackage{color}

\usepackage{ulem}

\usepackage{mathrsfs}

\usepackage{wasysym}

\usepackage{stackrel}

\usepackage{pgfplots}
\pgfplotsset{compat=newest}

\title{\large{\bf Geometric regularity estimates for quasi-linear elliptic models in non-divergence form with strong absorption}}
\author{\it by \smallskip \\ Claudemir Alcantara \footnote{\noindent Pontifícia Universidade Católica do Rio de Janeiro - PUC-Rio - Departamento de Matemática. Rio de Janeiro - RJ, Brazil \noindent \texttt{E-mail address: \url{alcantara@mat.puc-rio.br}}},\quad Jo\~{a}o Vitor da Silva
\footnote{\noindent Universidade Estadual de Campinas - UNICAMP. Departamento  de Matemática. Campinas - SP, Brazil. \noindent \texttt{E-mail address: \url{jdasilva@unicamp.br}}}\\$\&$\\ Ginaldo de Santana S\'{a}\footnote{\noindent Universidade Estadual de Campinas - UNICAMP. Departamento  de Matemática. Campinas - SP, Brazil. \noindent \texttt{E-mail address: \url{ginaldo@ime.unicamp.br}}}
}

\newlength{\hchng}
\newlength{\vchng}
\def \R {\mathbb{R}}

\def \dist {\mathrm{dist}}

\newcommand{\defeq}{\mathrel{\mathop:}=}
\newcommand{\NDelta}{\Delta_p^{\mathrm{N}}}

\newcommand{\intav}[1]{\mathchoice {\mathop{\vrule width 6pt height 3 pt depth  -2.5pt
\kern -8pt \intop}\nolimits_{\kern -6pt#1}} {\mathop{\vrule width
5pt height 3  pt depth -2.6pt \kern -6pt \intop}\nolimits_{#1}}
{\mathop{\vrule width 5pt height 3 pt depth -2.6pt \kern -6pt
\intop}\nolimits_{#1}} {\mathop{\vrule width 5pt height 3 pt depth
-2.6pt \kern -6pt \intop}\nolimits_{#1}}}

\newtheorem{theorem}{Theorem}[section]

\newtheorem{lemma}[theorem]{Lemma}

\newtheorem{proposition}[theorem]{Proposition}

\newtheorem{corollary}[theorem]{Corollary}

\theoremstyle{definition}

\newtheorem{definition}[theorem]{Definition}

\theoremstyle{remark}

\newtheorem{Assumption}{{\bf $\mathrm{A}$}}

\newtheorem{remark}[theorem]{Remark}

\numberwithin{equation}{section}

\begin{document}
\maketitle

\begin{abstract}

In this manuscript, we investigate geometric regularity estimates for problems governed by quasi-linear elliptic models in non-divergence form, which may exhibit either degenerate or singular behavior when the gradient vanishes, under strong absorption conditions of the form:
\[
|\nabla u(x)|^{\gamma} \Delta_p^{\mathrm{N}} u(x) = f(x, u) \quad \text{in} \quad B_1,
\]
where \(\gamma > -1\), \(p \in (1, \infty)\), and the mapping \(u \mapsto f(x, u) \lesssim \mathfrak{a}(x) u_{+}^m\) (with \(m \in [0, \gamma + 1)\)) does not decay sufficiently fast at the origin. This condition allows for the emergence of plateau regions, i.e., a priori unknown subsets where the non-negative solution vanishes identically. We establish improved geometric \(\mathrm{C}^\kappa_{\text{loc}}\) regularity along the set \(\mathscr{F}_0 = \partial \{u > 0\} \cap B_1\) (the free boundary of the model) for a sharp value of \(\kappa \gg 1\), which is explicitly determined in terms of the structural parameters. Additionally, we derive non-degeneracy results and other measure-theoretic properties. Furthermore, we prove a sharp Liouville theorem for entire solutions exhibiting controlled growth at infinity.

\medskip
\noindent \textbf{Keywords}: Geometric regularity estimates, quasi-linear elliptic models, non-divergence form operators.
\vspace{0.2cm}
	
\noindent \textbf{AMS Subject Classification: Primary 35B65, 35J60, 35J70, 35J75; Secondary 35D40}
\end{abstract}

\newpage

\section{Introduction}

In this article, we investigate sharp regularity estimates for diffusion problems governed by quasi-linear operators in non-divergence form (with singular or degenerate characteristics) for which a Minimum Principle does not apply. Specifically, we focus on prototypical models arising in Tug-of-War games (stochastic models with noise), where the presence of dead core regions plays a fundamental role in understanding the nature of such models (cf. \cite{Attou18}, \cite{BlancRossi-Book}, \cite{Lewicka20}, and \cite{Parvia2024}).

The model under consideration is given by
\begin{equation}\label{Problem}
|\nabla u(x)|^\gamma \Delta_p^{\mathrm{N}} u(x) = f(x, u) \lesssim \mathfrak{a}(x)u_+^m(x) \quad \text{in} \quad \Omega,
\end{equation}
where 
\begin{itemize}
    \item[(\textbf{A0})] \( p \in (1, \infty) \) is a fixed parameter (related to the elliptic nature of the operator);
    \item[(\textbf{A1})] \( \gamma \in (-1, \infty) \) governs the singularity/degeneracy of the model;
    \item[(\textbf{A2})] \( m \in [0, \gamma + 1) \) represents the strong absorption exponent of the model.
\end{itemize}
Additionally, we impose the boundary condition
\begin{equation}\label{Bound-Cond}
 u(x) = g(x) \quad \text{on } \partial \Omega,
\end{equation}
where the boundary datum is assumed to be continuous and non-negative, i.e., \( 0 \leq g \in \mathrm{C}^0(\partial \Omega) \).

Throughout this manuscript, for a fixed \( p \in (1, \infty) \), the operator
\[
\Delta^{\mathrm{N}}_p u = |\nabla u|^{2-p} \operatorname{div}(|\nabla u|^{p-2} \nabla u) = \Delta u + (p-2) \left\langle D^2 u \frac{\nabla u}{|\nabla u|}, \frac{\nabla u}{|\nabla u|}\right\rangle
\]
denotes the \textit{Normalized \( p \)-Laplacian operator} or \textit{Game-theoretic \( p \)-Laplace operator}, which arises in certain stochastic models (see \cite{BlancRossi-Book}, \cite{Lewicka20}, and \cite{Parvia2024} for comprehensive surveys on Tug-of-War games and their associated PDEs). Furthermore, the fully nonlinear operator
\[
\Delta_{\infty}^{\mathrm{N}} u(x) \coloneqq \left\langle D^2 u \frac{\nabla u}{|\nabla u|}, \frac{\nabla u}{|\nabla u|}\right\rangle
\]
is referred to as the \textit{Normalized infinity-Laplacian}, which also plays a significant role in the interplay between nonlinear PDEs and Tug-of-War games (cf. \cite{BlancRossi-Book} and \cite{Parvia2024}).

We assert that the Normalized \( p \)-Laplacian operator is ``uniformly elliptic'' in the sense that
\[
\mathcal{M}^{-}_{\lambda, \Lambda}(D^2 u) \leq \Delta_p^{\mathrm{N}} u \leq \mathcal{M}^{+}_{\lambda, \Lambda}(D^2 u),
\]
where
\begin{equation}\label{Pucci}\tag{Pucci}
\mathcal{M}^{-}_{\lambda, \Lambda}(D^2 u) \coloneqq \inf_{\mathfrak{A} \in \mathcal{A}_{\lambda, \Lambda}} \text{Tr}(\mathfrak{A} D^2 u) \quad \text{and} \quad 
\mathcal{M}^{+}_{\lambda, \Lambda}(D^2 u) \coloneqq \sup_{\mathfrak{A} \in \mathcal{A}_{\lambda, \Lambda}} \text{Tr}(\mathfrak{A} D^2 u)
\end{equation}
are the \textit{Pucci extremal operators}, and \(\mathcal{A}_{\lambda, \Lambda}\) denotes the set of symmetric \( n \times n \) matrices whose eigenvalues lie within the interval \([\lambda^{\mathrm{N}}_p, \Lambda^{\mathrm{N}}_p]\). Indeed, the Normalized \( p \)-Laplacian operator can be expressed in the form
\[
\Delta_p^{\mathrm{N}} u = \mathrm{Tr}\left[\left(\mathrm{Id}_n + (p - 2) \frac{\nabla u}{|\nabla u|} \otimes \frac{\nabla u}{|\nabla u|}\right) D^2 u\right],
\]
from which it is straightforward to verify that 
\[
\lambda^{\mathrm{N}}_p = \min\{1, p - 1\} \quad \text{and} \quad \Lambda^{\mathrm{N}}_p = \max\{1, p - 1\}.
\]

Next, we detail our assumption, which pertains to the bounded \textit{Thiele modulus} or modulating function \(\mathfrak{a}\).

\begin{Assumption}[\textbf{Thiele modulus}]\label{Assumption_lambda}
We assume that \(\mathfrak{a}: \Omega \to \mathbb{R}_+\) (referred to as the \textit{Thiele modulus} or modulating function) is a continuous and bounded function (strictly positive and finite). Specifically, there exist constants \(\Lambda_0 \geq \lambda_0 > 0\) such that
\[
\Lambda_0 \ge \sup_{\Omega} \mathfrak{a}(x) \geq \mathfrak{a}(x) \ge \inf_{\Omega} \mathfrak{a} \geq \lambda_0.
\]
\end{Assumption}

In this framework, equation \eqref{Problem} is classified as an equation with a strong absorption condition, and the set \(\partial \{ u > 0 \} \cap \Omega\) constitutes the (physical) free boundary of the problem.

The mathematical analysis of equation \eqref{Problem} is significant not only due to its applications but also because of its connection to various mathematical free boundary problems extensively studied in the literature (cf. Alt–Phillips \cite{Alt-Phillips86}, D\'{i}az \cite{Diaz85}, Friedman–Phillips \cite{Fried-Phill84}, and Phillips \cite{Phillips83}, as well as Shahgholian et al. \cite{KKPS2000} and \cite{Lee-Shahg2003} for related problems with variational structure; additionally, for the non-variational counterpart of these models, see Da Silva et al. \cite{daSLR21} and Teixeira \cite{Teix16}).

A fundamental aspect of free boundary problems like \eqref{Problem} is determining the optimal regularity of viscosity solutions. For instance, if we fix \(0 < m < \gamma + 1\) and \(\mathrm{R} > r > 0\), the radial profile \(\omega: B_{\mathrm{R}}(x_0) \to \mathbb{R}^+\) given by
\begin{equation}\label{Radial-profile}
    \omega(x) = \omega_{\gamma}^{\mathrm{R}}(x) = \mathrm{c}_{n, \gamma, m, p} \left(|x - x_0| - r\right)^{\beta}_+, \quad \text{for} \quad \beta = \frac{\gamma + 2}{\gamma + 1 - m},
\end{equation}
where
\[
\mathrm{c}_{n, \gamma, m, p} = \mathrm{c}(n, \gamma, m, p) = \frac{1}{\beta^{\frac{\gamma + 1}{\gamma + 1 - m}}} \frac{1}{\left[n - 1 + (p - 1)(\beta - 1)\right]^{\frac{1}{\gamma + 1 - m}}} > 0
\]
is a suitable constant, is a viscosity solution to the equation
\[
-|\nabla \omega(x)|^{\gamma} \Delta_p^{\mathrm{N}} \omega(x) + \omega^q(x) = 0 \quad \text{in} \quad B_{\mathrm{R}}(x_0).
\]
It is well-known that, in general, solutions to \eqref{Problem} belong to \(C^{1,\kappa}_{\text{loc}}\) for some \(\kappa \in (0,1)\) (cf. Attouchi–Ruosteenoja \cite{Attou18} for further details). However, for this particular example, with \(\gamma > -1\) fixed, one observes that \(\omega \in C^{\beta}_{\text{loc}}\) at free boundary points, where
\begin{equation}\label{beta-exponent}
\beta = \beta(\gamma, m) := \frac{\gamma + 2}{\gamma + 1 - m}.
\end{equation}
Notably, when \(m > 0\), we obtain
\[
\beta(\gamma, m) = \frac{\gamma + 2}{\gamma + 1 - m} > 1 + \frac{1}{\gamma + 1} \quad \Rightarrow \quad \text{Gain of smoothness along the free boundary.}
\]
In other words, within a small radius, viscosity solutions satisfy the sharp growth estimates
\begin{equation}\label{Growth-Control}
\varrho^{\beta} \lesssim \sup_{B_{\varrho}(x_0)} \omega(x) \lesssim \varrho^{\beta}.
\end{equation}

\subsection{Main results and some consequences}

Inspired by the preceding motivation, we will demonstrate that bounded viscosity solutions of \eqref{Problem} satisfy a growth estimate of the form \eqref{Growth-Control} near free boundary points (see Theorems \ref{Improved Regularity} and \ref{Thm-Non-Deg} for further details).

Next, we establish an improved regularity estimate (near free boundary points) for solutions of \eqref{Problem}, a sort of $\beta-$power-like decay (for $\beta>1$ as in \eqref{beta-exponent}).

\begin{theorem}[{\bf Improved Regularity}]\label{Improved Regularity}
Suppose that $(\mathrm{A}0)-(\mathrm{A}2)$, and $\mathrm{A}{\ref{Assumption_lambda}}$ (upper bound) hold true, and let $u\in \mathrm{C}^0(B_1)$ be  a non-negative bounded viscosity solution to
\[
|\nabla u(x)|^{\gamma}\NDelta u(x) = \mathfrak{a}(x) u_+^m(x) \quad \text{in} \quad B_1
\]
and let $x_0\in \partial\{u>0\}\cap B_{1/2}$. 
 Then, for any point $z\in \{u>0\}\cap B_{1/2}$, there exists a universal constant\footnote{Throughout this work, we will refer to universal constants when they depend only on the dimension and structural parameters of the
problem, i.e., $n,\,p,\, \gamma,\, m$, and the bounds of $\mathfrak{a}$.} $\mathrm{C}_{\mathcal{G}}>0$ such that
\[
u(z)\leq \mathrm{C}_{\mathcal{G}}\|u\|_{L^{\infty}(B_1)}|z-x_0|^{\beta}.
\]
\end{theorem}

As a consequence of our findings, we establish a more precise control for viscosity solutions of \eqref{Problem} near their free boundaries. This information is crucial for various quantitative aspects of numerous free boundary problems (see, e.g., \cite{ART17} and \cite{Teix18}). More precisely, we prove that, close to their free boundaries, solutions decay at a rate proportional to an appropriate power of \(\mathrm{dist}(\cdot, \partial\{u > 0\})\).

\begin{corollary}\label{Cor1.4}
Let \( u \) be a nonnegative, bounded viscosity solution to \eqref{Problem} in \( \Omega \). Given \( x_0 \in \{ u > 0 \} \cap \Omega' \) with \( \Omega^{\prime} \Subset \Omega \), then  
\[
u(x_0) \leq \mathrm{C}_{\sharp} \, \mathrm{dist} \left( x_0, \partial \{ u > 0 \} \right)^{\beta},
\]
where \( \mathrm{C}_{\sharp} > 0 \) is a universal constant.
\end{corollary}

\begin{remark}
It is worth emphasizing that Theorem \ref{Improved Regularity} holds for any singular (\(\gamma \in (-1, 0)\)) or degenerate (\(\gamma > 0\)) behavior of the toy model \eqref{Problem} (cf. \cite{daSLR21} for related results). However, when \(\gamma < 0\), we may rewrite (via Lemma \ref{Lemma2.5}) our model PDE as follows:
\[
\Delta_p^{\mathrm{N}} u(x) = f_0(x, u, \nabla u) \lesssim \mathfrak{a}(x) u_+^m(x) |\nabla u(x)|^{-\gamma},
\]
which can be interpreted as a non-variational counterpart of quasi-linear models studied by Teixeira in \cite{Tei22}.

As motivation, problems with sub-linear gradient growth are often associated with a class of singular equations in the sense that they degenerate when the gradient vanishes. For instance, by \cite[Lemmas 5.2 and 5.3]{KoKo17}, a suitable viscosity solution of
\[
\mathcal{M}_{\lambda, \Lambda}^+(D^2 u) + f_0(x) |Du|^{\gamma_{\textrm{S}}} = 0, \quad \text{with } 0 < \gamma_{\textrm{S}} < 1 \text{ and } f_0 \in C^0(\Omega) \cap L^p(\Omega),
\]
is also a viscosity solution of
\[
|Du|^{-\gamma_{\textrm{S}}} \mathcal{M}_{\lambda, \Lambda}^+(D^2 u) + f_0(x) = 0
\]
(see \cite[Section 7.3]{daSNorb21} for details). Consequently, such problems can be viewed as a fully nonlinear counterpart of the singular \((\gamma_{\textrm{S}} + 1)\)-Laplace-type models, highlighting their relevance.

In conclusion, this transformation allows the singular case to be reformulated as the analysis of the normalized \(p\)-Laplacian problem incorporating a term with a ``sub-linear regime'' (from a homogeneity perspective).
\end{remark}
\medskip

In our second result, we rigorously determine the sharp asymptotic behavior with which non-negative viscosity solutions of \eqref{Problem} detach from their dead core regions. This information is essential for various free boundary problems (see \cite{KKPS2000}, \cite{Lee-Shahg2003}, and \cite{Teix16}) and plays a pivotal role in establishing several weak geometric properties.

\begin{theorem}[{\bf Non-degeneracy}]\label{Thm-Non-Deg}
Let us assume that $(\mathrm{A}0)-(\mathrm{A}2)$, and $\mathrm{A} {\ref{Assumption_lambda}}$ (lower bound) hold, and let $u\in \mathrm{C}^0(B_1)$ be a non-negative bounded viscosity solution to 
\[
|\nabla u(x)|^{\gamma}\NDelta u(x)=\mathfrak{a}(x)u^m_+(x) \quad \text{in} \quad B_1,
\]
and let $y\in \overline{\{u>0\}}\cap B_{1/2}$. Then, there exists a universal positive constant $\mathrm{C}_{\mathrm{ND}}$ such that for $r\in (0,1/2)$ hold
\[
\sup_{\overline{B}_r(y)}u\geq \mathrm{C}_{\mathrm{ND}}(n,\gamma, m, p, \lambda_0)\cdot r^{\beta},
\] 
where 
\begin{equation}\label{Const-Radial-Profile}
\mathrm{C}_{\mathrm{ND}}(\gamma, n, m, p, \lambda_0) = \frac{1}{\beta^{\frac{\gamma+1}{\gamma+1-m}}}\left[\frac{\lambda_0}{n-1+(p-1)(\beta-1)}\right]^{\frac{1}{\gamma+1-m}} > 0.\quad 
\end{equation}
\end{theorem}

We now establish certain measure-theoretic properties of the phase transition. The following result demonstrates that the region \(\{u > 0\}\) maintains a uniform positive density along the free boundary. In particular, the formation of cusps at free boundary points is precluded.

\begin{corollary}[{\bf Uniform positive density}]\label{Corollary-UPD}
Assume that $\mathrm{A} {\ref{Assumption_lambda}}$ hold true, let $u\in \mathrm{C}^0(B_1)$ be a non-negative bounded viscosity solution to
\[
|\nabla u(x)|^{\gamma}\NDelta u(x)=\mathfrak{a}(x)u^m_+(x) \quad \text{in} \quad B_1,
\]
and $x_0\in\partial\{u>0\}\cap B_{1/2}$. Then, for all $r\in (0,1/2)$ holds 
\[
\mathscr{L}^n( B_r(x_0)\cap\{u>0\})\geq \theta r^n,
\]
where $\theta$ is a positive universal constant, and $\mathscr{L}^n(\mathrm{E})$ states the $n$-dimensional Lebesgue measure of set $\mathrm{E}$.
\end{corollary}

An important result following from Theorem \ref{Thm-Non-Deg} is the refinement of the growth behavior near the free boundary. Specifically, given \( x_0 \in \{ u > 0 \} \cap \Omega \), it holds that
\[
\mathrm{dist}(x_0, \partial \{ u > 0 \})^{\beta} \lesssim u(x_0).
\]

\begin{corollary}\label{Corollary-Non-Deg}  Let \( u \) be a non-negative, bounded viscosity solution to \eqref{Problem} in \( \Omega \) and \( \Omega' \subset \Omega \).  
Given \( x_0 \in \{ u > 0 \} \cap \Omega \), there exists a universal constant \( \mathrm{C}_{\ast} > 0 \) such that  
\[
u(x_0) \geq \mathrm{C}_{\ast} \, \mathrm{dist}(x_0, \partial \{ u > 0 \})^{\beta}.
\]
\end{corollary}

As a result of Theorems \ref{Improved Regularity} and \ref{Thm-Non-Deg}, we also establish the porosity of the zero-level set.

\begin{definition}[{\bf Porous set}] 
A set \( \mathcal{A} \subset \mathbb{R}^n \) is said \textit{porous} with porosity \( \delta_{\mathcal{A}} > 0 \) if there exists \( \mathrm{R} > 0 \) such that 
\[ 
\forall x \in \mathcal{A}, \, \forall\, r \in (0, \mathrm{R}), \, \exists\, y \in \mathbb{R}^n \, \text{such that} \; B_{\delta_{\mathcal{A}} r}(y) \subset B_r(x) \setminus \mathcal{A}.
\] 
\end{definition} 

We note that a porous set with porosity \( \delta_{\mathcal{A}} > 0 \) satisfies
$$
\mathscr{H}_{\mathrm{dim}}(\mathcal{A}) \leq n - \mathrm{c} \delta_{\mathcal{A}}^n,
$$
where \( \mathscr{H}_{\mathrm{dim}} \) states the Hausdorff dimension and \( \mathrm{c} = \mathrm{c}(n) > 0 \) is a dimensional constant. Particularly, a porous set has Lebesgue measure zero (see, for example, \cite{KR97} and \cite{Zaj87}).

\begin{corollary}[{\bf Porosity of the free boundary}]\label{Corollary-Porosity}
There exists a universal constant $\tau\in(0,1]$ such that
\[
\mathscr{H}^{n-\tau}\left(\partial\{u>0\}\cap B_{1/2}\right)<+\infty.
\]
\end{corollary}

We also establish a precise and rigorous characterization of solutions to \eqref{Problem} when $m = \gamma + 1$. It is worth noting that, in this case, the previously derived regularity estimates deteriorate. Consequently, investigating this critical regime constitutes a subtle and challenging endeavor.

\begin{theorem}[{\bf Strong Maximum Principle}]\label{Thm-SMP}
Let $u$ be a nonnegative, bounded viscosity solution to \eqref{Problem} with $m = \gamma + 1$. Then, the following dichotomy holds: either $u > 0$ or $u \equiv 0$ in $\Omega$.
\end{theorem}

Next, we obtain the sharp (and improved) rate, in which the gradient of solutions decays at interior free boundary points.

\begin{theorem}[{\bf Sharp gradient decay}]\label{Thm-Sharp-Grad}
Let $u$ be a bounded, nonnegative viscosity solution to \eqref{Problem} (with $\gamma>0$). Then, for any point $z \in \partial\{u > 0\} \cap \Omega^{\prime}$ with $\Omega^{\prime} \Subset \Omega$, there exists a universal constant $\mathrm{C}^{^{\prime}}_0 > 0$ such that  
\[
\sup_{B_r(z)} |\nabla u(x)| \leq \mathrm{C}_0^{^{\prime}} r^{\frac{1+m}{\gamma+1-m}} \quad \text{for all} \quad  
0 < r \ll \min \left\{ 1, \frac{\mathrm{dist}(\Omega^{^{\prime}}, \partial\Omega)}{2} \right\}.
\]
\end{theorem}

As a consequence, we derive the following estimate in terms of the distance up to the free boundary (see also \cite[Corollary 5.1]{daSS18}. The proof follows a similar reasoning to the one addressed in Corollary \ref{Cor1.4}.

\begin{corollary}
Let \( u \) be a bounded, non-negative viscosity solution to \eqref{Problem} in \( B_1 \) (with $\gamma>0$). Then, for any point \( z \in \{ u > 0 \} \cap B_{1/2} \), there exists a universal constant \( \mathrm{C}_{\star} > 0 \) such that  
\[
|\nabla u(z)| \leq \mathrm{C}_{\star} \, \mathrm{dist} (z, \partial \{ u > 0 \})^{\frac{1+m}{\gamma+1-m}}.
\]
\end{corollary}

The methodology adopted in our article is sufficiently flexible to generate several other noteworthy applications related to regularity at free boundary points. Among these, we establish a novel estimate in the \( L^2 \)-average sense for ``\( (\gamma+1) \)-dead core solutions.''. More specifically, we have the following estimate:

\begin{theorem}[{\bf Estimate in \( L^2 \)-average}]\label{Thm-Aver-L2} For every normalized viscosity solution \( u \in \mathrm{C}^0(B_1) \) of \eqref{Problem} with \( \gamma > 0 \) and \( x_0 \in \partial\{u > 0\} \cap B_{1/2} \),  
there exists \( \mathrm{M}_0 = \mathrm{M}_0(n, p, m, \gamma, \lambda_0, \Lambda_0) \) such that  
\[
\left(\intav{B_r(x_0)} (|\nabla u(x)|^{\gamma} |D^2 u(x)|)^2 \, dx \right)^{\frac{1}{2}}\leq \mathrm{M}_0 r^{\frac{\gamma m}{\gamma+1-m}}.
\]
\end{theorem}

Next, as an application, a qualitative Liouville-type result for entire solutions is addressed provided that their growth at infinity can be controlled appropriately (see \cite[Theorem 3.1]{BirDem} for related results).

\begin{theorem}[{\bf Liouville-type Theorem I}]\label{Thm Liouv-I}
Suppose that $\mathrm{A} {\ref{Assumption_lambda}}$ holds, and let $u$ be a non-negative viscosity solution of
\[
|\nabla u(x)|^{\gamma}\NDelta u(x)=\mathfrak{a}(x)u^m_+(x) \quad \text{in} \quad \mathbb{R}^n,
\]
and $u(x_0)=0$. If 
\begin{equation}\label{o-pequeno}
u(x)=\text{o}\left(|x|^{\beta}\right), \quad \text{as} \quad |x|\rightarrow \infty,
\end{equation}
then $u\equiv 0$.
\end{theorem}

Finally, we establish a sharp Liouville-type result, which asserts that any entire viscosity solution satisfying a controlled growth condition at infinity must be identically zero.

\begin{theorem}[{\bf Liouville-type Theorem II}]\label{Thm Liouv-II}
Let \( u \) be a non-negative viscosity solution to 
\[
|\nabla u(x)|^{\gamma}\NDelta u(x)=\mathfrak{a}(x)u^m_+(x) \quad \text{in} \quad \mathbb{R}^n.
\]
Then, \( u \equiv 0 \) provided that
\[
\limsup_{|x| \to \infty} \frac{u(x)}{|x|^{\beta}} <\mathrm{C}_{\mathrm{ND}}(\gamma, n, m, p, \lambda_0),
\]
where $\mathrm{C}_{\mathrm{ND}}(\gamma, n, m, p, \lambda_0)>0$ comes from \eqref{Const-Radial-Profile}.
\end{theorem}

\medskip

\subsection{Motivation and State-of-the-Art}

\subsubsection*{How can we find the $\Delta_p^{\mathrm{N}}(\cdot)$? A journey through game theoretic methods in PDEs.
}

Over the past two decades, a connection between stochastic Tug-of-War games and nonlinear partial differential equations of \( p \)-Laplacian type, initiated by the pioneering work of Peres and Sheffield \cite{PerShef08}, has attracted significant attention and introduced new trends in the application of game-theoretical methods to PDEs.

We recall that elliptic PDE models governed by normalized operators with an elliptic nature have gained increasing interest in recent years due to their broad connections with:
\begin{enumerate}
    \item[\checkmark] Stochastic processes (random walks);
    \item[\checkmark] Probability theory (dynamic programming principles);
    \item[\checkmark] Game theory (Tug-of-War games with noise);
    \item[\checkmark] Analysis (asymptotic mean value characterizations of solutions to certain nonlinear PDEs),
\end{enumerate}
among others (cf. \cite{AHP21}, \cite{AP20}, \cite{Arro-Parv24}, \cite{BCMR22}, \cite{BlancRossi-Book}, \cite{Lewicka20}, \cite{MPR12}, \cite{Parvia2024}, and \cite{Rossi11} for insightful contributions).

As motivation, we will discuss \textbf{stochastic Tug-of-War games with noise}, which model random processes where two players compete to optimize their outcomes (cf. \cite{Lewicka20} and \cite{Parvia2024} for detailed mathematical insights). Specifically, we introduce two competing players: A token is initially placed at \( x_0 \in \Omega \subset \mathbb{R}^n \). A biased coin with probabilities \( \alpha_0 \) and \( \beta_0 \) (where \( \alpha_0 + \beta_0 = 1 \), \( \beta_0 > 0 \), and \( \alpha_0 \geq 0 \)) is then tossed. If the outcome is headed (with probability \( \beta_0 \)), the next position \( x_1 \in B_{\varepsilon}(x_0) \) is chosen according to the uniform probability distribution over the ball \( B_{\varepsilon}(x_0) \), as in the random walk on balls.

Conversely, if the outcome is tails (with probability \( \alpha_0 \)), a fair coin is flipped, and the winner of the toss moves the token to a new position \( x_1 \in B_\varepsilon(x_0) \). The players continue this process until the token exits the bounded domain \( \Omega \). At the conclusion of the game, Player II compensates Player I with an amount determined by the payoff function \( \mathrm{F}(x_\tau) \), where \( x_\tau \) is the first position outside \( \Omega \).

The value of the game is defined as the expected payoff \( \mathrm{F}(x_\tau) \), considering that Player I aims to maximize the outcome while Player II seeks to minimize it. More precisely, the value of the game for Player I is given by:

$$
    u^{\varepsilon}_{\mathrm{I}}(x_0) = \sup_{\mathrm{S}_{\mathrm{I}}} \inf_{\mathrm{S}_{\mathrm{II}}} \mathbb{E}_{\mathrm{S}_{\mathrm{I}}, \mathrm{S}_{\mathrm{II}}}^{x_0} [\mathrm{F}(x_\tau)],
$$
while the value of the game for Player II is expressed as:
$$
    u^{\varepsilon}_{\mathrm{II}}(x_0) = \inf_{\mathrm{S}_{\mathrm{II}}} \sup_{\mathrm{S}_{\mathrm{I}}} \mathbb{E}_{\mathrm{S}_{\mathrm{I}}, \mathrm{S}_{\mathrm{II}}}^{x_0} [\mathrm{F}(x_\tau)],
$$
where $x_0$ represents the initial position of the game, and $\mathrm{S}_{\mathrm{I}}$ and $\mathrm{S}_{\mathrm{II}}$ denote the strategies adopted by the players, respectively.

In such a context, the value function satisfies the \textbf{Dynamic Programming Principle} (DPP, for short):

\begin{equation}\tag{{\bf DPP}}
    u_\varepsilon(x) = \frac{\alpha_0}{2} \left\{ \sup_{B_\varepsilon(x)} u_\varepsilon + \inf_{B_\varepsilon(x)} u_\varepsilon \right\} + \beta_0 \intav{B_\varepsilon(0)} u_\varepsilon(x + h) \,d\mathcal{L}^n ,
\end{equation}
where
$$
    u_\varepsilon = \mathrm{F} \quad \text{on} \quad \Gamma_\varepsilon := \{ x \in \mathbb{R}^n \setminus \Omega : \operatorname{dist}(x, \partial \Omega) \leq \varepsilon \}
$$
and
$$
    \alpha_0 = \frac{p - 2}{p + n} \quad \text{and} \quad \beta_0 = 1 - \alpha_0 = \frac{2 + n}{p + n} \quad (\text{for} \quad p>2).
$$

Note that the DPP can be interpreted as a discretization of an elliptic PDE, which will be introduced shortly.

\begin{theorem}[{\bf \cite[Theorem 3.9]{Lewicka20} and \cite[ Theorem 3.1]{PerShef08}}]\label{Thm3.1-Parviainen}
    Given a bounded Borel boundary function $\mathrm{F}:  \Gamma_{\epsilon} \to  \mathbb{R}$, there is a bounded Borel function $u$ that satisfies the $\mathrm{DPP}$ with the boundary values $\mathrm{F}$.
\end{theorem}

\begin{theorem}[{\bf \cite[Theorem 3.14]{Lewicka20} and \cite[Theorem 3.2]{PerShef08} - Value of the Game}] It holds that $u =  u^{\varepsilon}_{\mathrm{I}} = u^{\varepsilon}_{\mathrm{II}}$, where $u$ is a solution of the $\mathrm{DPP}$ as in Theorem \ref{Thm3.1-Parviainen}. In other words, the game has a unique value.
\end{theorem}

Finally, we can determine the PDE satisfied by a uniform limit derived from the DPP. For this reason, the limiting operator is referred to as the \textit{Game-theoretic \( p \)-Laplace operator}.

\begin{theorem}[{\bf \cite[Theorem 5.7]{Lewicka20} and  \cite[ Theorem 3.3]{PerShef08}}]
    Let $\Omega$ be a smooth domain, and let $\mathrm{F}: \mathbb{R}^n \setminus \Omega \to \mathbb{R}$ be a smooth function. Let $u$ be the unique viscosity solution to the problem:
$$
\left\{
\begin{array}{rcrcl}
\Delta_p^{\mathrm{N}} u & = & 0 & \text{in } &\Omega, \\
u & = &\mathrm{F} & \text{on} &\partial \Omega.
\end{array}
\right.
$$
Furthermore, let $u_\varepsilon$ be the value function for the Tug-of-War with noise with the boundary payoff function $\mathrm{F}$. Then,
\begin{equation}
u_\varepsilon \to u \quad \text{uniformly on } \overline{\Omega}, \quad \text{as } \quad  \varepsilon \to 0.
\end{equation}
\end{theorem}

Regarding regularity estimates, Luiro \textit{et al.} in \cite{LPS13} introduced a novel and relatively straightforward proof of Harnack's inequality and Lipschitz continuity for \( p \)-harmonic functions, where \( p > 2 \). The proof is based on the strategic selection of strategies for the stochastic Tug-of-War game, making it entirely distinct from the classical proofs established by De Giorgi and Moser. The methodology relies on a recently discovered connection between stochastic games and \( p \)-harmonic functions, as elucidated by Peres and Sheffield in their seminal work \cite{PerShef08}.

Subsequently, Luiro and Parviainen in \cite{LuiroParv18} established regularity properties for functions satisfying a specific dynamic programming principle. Such functions may arise, for instance, in the context of stochastic games or discretization methods. Their results can also be applied to derive regularity and existence results for the corresponding PDEs.

Finally, concerning recent trends in game-theoretical methods, Arroyo and Parviainen in \cite{Arro-Parv24} introduced a novel formulation of the Tug-of-War game, characterized by a dynamic programming principle (DPP) associated with the \( p \)-Laplacian for \( p \in (1, 2) \). In this setting, the asymptotic H\"older continuity of solutions can be directly deduced from recent Krylov-Safonov-type regularity results in the singular case. Additionally, the existence and uniqueness of solutions to the DPP are established. Furthermore, the regularity properties and the connection between this DPP (and its corresponding Tug-of-War game) and the \( p \)-Laplacian are thoroughly investigated.

\subsubsection*{PDEs models with the presence of dead cores - old and recent trends}

In the following, we will highlight some pivotal contributions concerning non-uniformly elliptic dead core models (in both divergence and non-divergence forms) over the past decades.

Several elliptic models with free boundaries arise in various phenomena associated with reaction-diffusion and absorption processes, both in pure and applied sciences. Notable examples include models in chemical and biological processes, combustion phenomena, and population dynamics, among other applications. A particularly significant problem within this framework pertains to diffusion processes under a sign constraint—the well-known one-phase problems—which, in chemical and physical contexts, represent the only relevant case to be considered (cf. \cite{Aris75-I, Aris75-II, Band-Sperb-Stak84}, and \cite{Diaz85} and references therein for further motivation). As motivation, a class of such problems is given by
\[
\left\{
\begin{array}{rclcl}
-\mathcal{Q}u(x) + f(u)\chi_{\{u > 0\}}(x) & = & 0 & \text{in } & \Omega \\
u(x) & = & g(x) & \text{on } & \partial\Omega,
\end{array}
\right.
\]
where \(\mathcal{Q}\) denotes a quasi-linear elliptic operator in divergence form with \( p \)-structure for \( 2 \leq p < \infty \) (cf. \cite{Choe91} and \cite{Serrin64} for further details), and \(\Omega \subset \mathbb{R}^n\) is a regular and bounded domain. In this setting, \( f \) is a continuous and monotonically increasing reaction term satisfying \( f(0) = 0 \), and \( 0 \leq g \in \mathrm{C}^0(\partial\Omega) \). In models from applied sciences, \( f(u) \) represents the ratio of the reaction rate at concentration \( u \) to that at concentration one.

It is worth noting that when the nonlinearity \( f \in C^1(\Omega) \) is locally \((p - 1)\)-Lipschitz near zero (i.e., \( f \) satisfies a Lipschitz condition of order \((p - 1)\) at \( 0 \) if there exist constants \(\mathrm{M}_0, \zeta > 0\) such that \( f(u) \leq \mathrm{M}_0 u^{p - 1} \) for \( 0 < u < \zeta \)), it follows from the Maximum Principle that non-negative solutions must, in fact, be strictly positive (cf. \cite{Vazq84}).

However, the function \( f \) may fail to be differentiable or to decay sufficiently fast at the origin. For instance, if \( f(t) \) behaves as \( t^q \) with \( 0 < q < p - 1 \), then \( f \) does not satisfy the Lipschitz condition of order \((p - 1)\) at the origin. In this scenario, problem (1.1) lacks the Strong Minimum Principle, meaning that non-negative solutions may completely vanish within an a priori unknown region of positive measure \(\Omega_0 \subset \Omega\), referred to as the \textit{Dead Core} set (cf. D\'iaz’s monograph \cite{Diaz85}, Chapter 1, for a comprehensive survey on this topic).

As an illustration of the aforementioned discussion (cf. \cite{Aris75-I}, \cite{Aris75-II}, and \cite{Stak86}), for a domain \(\Omega \subset \mathbb{R}^n\), certain (stationary) isothermal and irreversible catalytic reaction processes can be mathematically described by boundary value problems of the reaction-diffusion type, given by
\[
-\Delta u(x) + \lambda_0(x) u_+^q(x) = 0 \quad \text{in } \Omega, \quad \text{and} \quad u(x) = 1 \quad \text{on } \partial\Omega,
\]
where, in this context, \( u \) represents the concentration of a chemical reagent (or gas), and the non-Lipschitz kinetics corresponds to the \( q \)-th-order Freundlich isotherm. Furthermore, \(\lambda_0 > 0\), known as the \textit{Thiele Modulus}, governs the ratio between the reaction rate and the diffusion-convection rate.

As previously mentioned, when \( q \in (0, 1) \), the strong absorption due to chemical reactions may outpace the supply driven by diffusion across the boundary, potentially causing the complete depletion of the chemical reagent in certain sub-regions, known as dead core zones, mathematically represented as \(\Omega_0:= \{x \in \Omega: u(x) = 0\} \subset \Omega\). In these regions, no chemical reaction occurs. Consequently, understanding the qualitative and quantitative behavior of dead core solutions is of paramount importance in chemical engineering and other applied sciences.

More specifically, for an arbitrary domain \(\Omega \subset \mathbb{R}^n\), the manuscript \cite{PS06} is dedicated to the study of various phenomena related to dead cores and bursts in quasi-linear partial differential equations of the form:
\begin{equation}
\mathrm{div} \big(\mathfrak{A}(|\nabla u|) \nabla u \big) = f(u) \quad \text{in } \quad \Omega,
\end{equation}
where \(\mathfrak{A} \in \mathrm{C}^0((0, \infty))\) satisfies the conditions that the mapping \(\rho \mapsto \rho \mathfrak{A}(\rho)\) is increasing and \(\displaystyle \lim_{\rho \to 0} \rho \mathfrak{A}(\rho) = 0\). Moreover, the nonlinearity \( f \) is assumed to be continuous, satisfying \( f(0) = 0 \), non-decreasing on \(\mathbb{R}\), and strictly positive on \((0, \infty)\).

This work extends the results obtained in \cite{PucciSerrin2004}, where Pucci and Serrin investigated similar problems. The main contributions of the present paper establish the existence of solutions exhibiting dead cores. Additionally, the authors demonstrate that, under certain conditions, solutions may exhibit both a dead core and bursts within the core.

Recently, in the paper \cite{daSS18}, Da Silva and Salort investigated diffusion problems governed by quasi-linear elliptic models of the \( p \)-Laplace type, for which a minimum principle is unavailable. Specifically, they considered the following boundary value problem:

\begin{equation}
\left\{
\begin{array}{rclccc}
-\text{div}(\Phi(x, u, \nabla u)) + \lambda_0(x) f(u) \chi_{\{u > 0\}} &=& 0 & \text{in } & \Omega, \\[0.2cm]
u(x) &=& g(x) & \text{on } & \partial \Omega,
\end{array}
\right.
\end{equation}
where \(\Omega \subset \mathbb{R}^n\) is a bounded domain, \(\Phi: \Omega \times \mathbb{R}_+ \times \mathbb{R}^n \to \mathbb{R}^n\) is a continuous, monotone vector field satisfying \( p \)-ellipticity and \( p \)-growth conditions, \( 0 \leq g \in \mathrm{C}^0(\partial \Omega) \), \(\lambda_0 \in \mathrm{C}^0(\overline{\Omega})\) is a non-negative bounded function, and \( f \) is a continuous, increasing function with \( f(0) = 0 \). This model allows for the existence of solutions with dead core zones, i.e., a priori unknown regions where non-negative solutions vanish identically.

In this setting, Da Silva and Salort established sharp and improved \( C_{\text{loc}}^{\beta(p, q)} \) regularity estimates along free boundary points, which are points belonging to \(\partial \{ u > 0 \} \cap \Omega\), where the regularity exponent is explicitly given by
\[
\beta(p, q) = \frac{p}{p - 1 - q}.
\]

The precise statement of these regularity results is given by the following estimate:
\[
\sup_{B_r(x_0)} u(x) \leq \mathrm{C}_0 \max \left\{ \inf_{B_r(x_0)} u(x),\,\, r^{\beta(p, q)} \right\},
\]
for a positive universal constant \(\mathrm{C}_0\) depending on \( n, p, q, \lambda_0 \). In particular, if \( x_0 \in \partial \{ u > 0 \} \cap \Omega \) (i.e., \( x_0 \) is a free boundary point), then
\[
\sup_{B_r(x_0)} u(x) \leq \mathrm{C}_0 r^{\beta(p, q)},
\]
for all \( 0 < r < \min \{1, \text{dist}(x_0, \partial \Omega)/2 \} \).

Additionally, several geometric and measure-theoretic properties, such as non-degeneracy, uniform positive density, and porosity of the free boundary, are also addressed in \cite{daSS18}. As an application, a Liouville-type theorem is established for entire solutions whose growth at infinity is appropriately controlled.

Subsequently, in the manuscript \cite{daSRS19}, Da Silva \textit{et al.} investigated the regularity properties of dead core problems and their limiting behavior as \( p \to \infty \) for elliptic equations of the \( p \)-Laplacian type, where \( 2 < p < \infty \), under a strong absorption condition. Specifically, the authors consider the problem:
\begin{equation}\label{p-dead core}
\left\{
\begin{array}{rclcc}
-\Delta_p u(x) + \lambda_0(x) u^q_+(x) &= & 0 & \text{in } & \Omega, \\[0.2cm]
u(x) &=& g(x) & \text{on} & \partial \Omega,
\end{array}
\right.
\end{equation}
where \( 0 \leq q < p - 1 \), \( 0 < \inf_{\Omega} \lambda_0 \leq \lambda_0 \) is a bounded function, \( \Omega \subset \mathbb{R}^n \) is a bounded domain, and \( g \in \mathrm{C}^0(\partial \Omega) \). The PDE in \eqref{p-dead core} is characterized by a strong absorption condition, and the set \( \partial \{ u > 0 \} \cap \Omega \) constitutes the free boundary of the problem.

In this context, the authors established the strong non-degeneracy of a solution \( u \) at points in \( \{ u > 0 \} \cap \Omega \) (see \cite[Theorem 1.1]{daSRS19}). Furthermore, they provided improved regularity along the free boundary \( \partial \{ u > 0 \} \cap \Omega \) for a fixed pair \( p \in (2, \infty) \) and \( q \in [0, p - 1) \) (see \cite[Theorem 1.2]{daSRS19}).

In the manuscript \cite{daSLR21}, Da Silva \textit{et al.} investigated reaction-diffusion problems governed by the second-order nonlinear elliptic equation:
\begin{equation}\label{Eq-DC-GeneralFN}
F(x, \nabla u, D^2 u) + |\nabla u|^{\gamma} \langle \vec{b}(x), \nabla u \rangle = \lambda_0(x) u^{\mu} \chi_{\{u > 0\}}(x), \quad \text{in } \Omega,
\end{equation}
where \(\Omega \subset \mathbb{R}^n\) is a smooth, open, and bounded domain, \( g \geq 0 \), \( g \in \mathrm{C}^0(\partial \Omega) \), \(\vec{b} \in \mathrm{C}^0(\overline{\Omega}, \mathbb{R}^n)\), \( 0 \leq \mu < \gamma + 1 \) is the absorption factor, \(\lambda_0 \in \mathrm{C}^0(\overline{\Omega})\), and \( F: \Omega \times (\mathbb{R}^n \setminus \{0\}) \times \text{Sym}(n) \to \mathbb{R} \) is a second-order fully nonlinear elliptic operator with measurable coefficients that satisfies certain ellipticity and homogeneity conditions. Specifically,
\[
|\vec{\xi}|^{\gamma} \mathcal{M}^{-}_{\lambda, \Lambda}(\mathrm{X} - \mathrm{Y}) \leq F(x, \vec{\xi}, \mathrm{X}) - F(x, \vec{\xi}, \mathrm{Y}) \leq |\vec{\xi}|^{\gamma} \mathcal{M}^{+}_{\lambda, \Lambda}(\mathrm{X} - \mathrm{Y}) \quad (\text{with } \gamma > -1)
\]
for all \(\mathrm{X}, \mathrm{Y} \in \text{Sym}(n)\) (the space of \( n \times n \) symmetric matrices) with \(\mathrm{X} \geq \mathrm{Y}\), where \(\mathcal{M}^{\pm}_{\lambda, \Lambda}(\cdot)\) denote the \textit{Pucci extremal operators}, defined previously in \eqref{Pucci}.

The authors establish the sharp asymptotic behavior governing the dead core sets of non-negative viscosity solutions to \eqref{Eq-DC-GeneralFN}. Furthermore, they derive a sharp regularity estimate at free boundary points, i.e., \(\mathrm{C}_{\text{loc}}^{\frac{\gamma + 2}{\gamma + 1 - \mu}}\). Additionally, they obtain the exact rate at which the gradient decays at interior free boundary points, along with a sharp Liouville-type result (cf. \cite{Teix16} for the corresponding results in the case \(\gamma = 0\) and \(\vec{b} = \vec{0}\)).

Finally, the last relevant free boundary model to highlight concerns the \(\infty\)-dead core problem. Specifically, in \cite{ALT16}, Ara\'{u}jo \textit{et al.} focused on investigating reaction-diffusion models governed by the \(\infty\)-Laplacian operator. For \(\lambda > 0\), \(0 \leq \gamma < 3\), and \(0 < \phi \in C^0(\partial \Omega)\), let \(\Omega \subset \mathbb{R}^n\) be a bounded open domain, and define
\begin{equation}\label{EqDeadCore}
\mathcal{L}^{\gamma}_{\infty} v := \Delta_{\infty} v - \lambda \cdot (v_+)^{\gamma} = 0 \quad \text{in} \quad \Omega \quad \text{and} \quad v = \phi \quad \text{on} \quad \partial \Omega.
\end{equation}
Here, the operator \(\mathcal{L}^{\gamma}_{\infty}\) represents the \(\infty\)-diffusion operator with \(\gamma\)-strong absorption, and the constant \(\lambda > 0\) is referred to as the \textit{Thiele modulus}, which controls the ratio between the reaction rate and the diffusion-convection rate (cf. \cite{Diaz85} for related results).

We must also emphasize that a significant feature in the mathematical representation of equation \eqref{EqDeadCore} is the potential occurrence of plateaus, namely sub-regions, \textit{a priori} unknown, where the function becomes identically zero (see \cite{daSRS19}, \cite{daSS18}, and \cite{Diaz85} for corresponding quasi-linear dead core problems).

After proving the existence and uniqueness of viscosity solutions via Perron's method and a Comparison Principle (see \cite[Theorem 3.1]{ALT16}), the main result in the manuscript \cite{ALT16} guarantees that a viscosity solution is pointwise of class \( C^{\frac{4}{3-\gamma}} \) along the free boundary of the non-coincidence set, i.e., \( \partial \{ u > 0 \} \) (see \cite[Theorem 4.2]{ALT16}). This implies that the solutions grow precisely as \( \mathrm{dist}^{\frac{4}{3-\gamma}} \) away from the free boundary. Additionally, by utilizing barrier functions, the authors demonstrate that such an estimate is sharp in the sense that \( u \) separates from its coincidence region precisely as \( \mathrm{dist}^{\frac{4}{3-\gamma}} \) (see, e.g., \cite[Theorem 6.1]{ALT16}).

Furthermore, the authors established some Liouville-type results. Specifically, if \( u \) is an entire viscosity solution to
\[
\Delta_{\infty} u(x) = \lambda u^{\gamma}(x) \quad \text{in} \quad \mathbb{R}^n,
\]
with \( u(0) = 0 \), and \( u(x) = \text{o}\left(|x|^{\frac{4}{3-\gamma}}\right) \), then \( u \equiv 0 \) (see \cite[Theorem 4.4]{ALT16}).

Finally, they also addressed a refined quantitative version of the previous result. Specifically, if \( u \) is an entire viscosity solution to
\[
\Delta_{\infty} u (x) = \lambda u^{\gamma}(x) \quad \text{in} \quad \mathbb{R}^n,
\]
such that
\[
\limsup_{|x| \to \infty} \frac{u(x)}{|x|^{\frac{4}{3-\gamma}}} < \left(\frac{\lambda(3-\gamma)^4}{4^3(1+\gamma)}\right)^{\frac{1}{3-\gamma}},
\]
then \( u \equiv 0 \) (see \cite[Theorem 5.1]{ALT16}).
\medskip

In conclusion, based on the literature review presented above, to the best of our knowledge, there are no research results addressing quasi-linear models in the non-divergence form such as \eqref{Problem}. This gap in the literature serves as one of the primary motivations for studying optimal regularity estimates in non-variational dead core problems with a singular or degenerate signature.

\section{Definitions and auxiliary results}
In this section, we introduce the concept of viscosity solutions, which will serve as the framework for our analysis, and highlight some auxiliary results essential for developing our work. We denote by \( B_r = B_r(0) \) the ball of radius \( r \) centered at \( x = \vec{0} \). In what follows, we define viscosity solutions.

\begin{definition}[{\bf Viscosity Solution - case $\gamma>0$}] 
An upper semicontinuous function \( u \in \mathrm{C}^0(\Omega) \) is called a viscosity subsolution (resp. supersolution) of
\begin{equation}\label{Vis1}
     |\nabla u|^{\gamma}\Delta_{p}^{\mathrm{N}} u = f(x, u(x)), \quad \text{with} \quad f \in \mathrm{C}^0(\Omega \times \mathbb{R}_+),
\end{equation}
if, whenever \( \phi \in C^{2}(\Omega) \) and \( u - \phi \) has a local maximum (resp. minimum) at \( x_{0} \in \Omega \), the following conditions are satisfied:

\begin{enumerate}
    \item If \( \nabla \phi(x_{0}) \neq 0 \), then 
    $$
      |\nabla \phi(x_0)|^{\gamma}\Delta_{p}^{\mathrm{N}} \phi(x_{0}) \geq f\big(x_{0}, \phi(x_0)\big) \quad  (\text{resp.}\,\,\,\,   |\nabla \phi(x_0)|^{\gamma}\Delta_{p}^{\mathrm{N}} \phi(x_{0})\leq f\big(x_{0}, \phi(x_0)\big) );
    $$

\item If \(\nabla \phi(x_{0}) = 0 \), then
\begin{enumerate}
\item For $p \geq 2$
$$
\Delta \phi(x_{0}) + (p - 2) \lambda_{\max}\big(D^{2} \phi(x_{0})\big) \geq f\big(x_{0}, \phi(x_0)\big)
$$
$$
(\text{resp.}\,\,\, \Delta \phi(x_{0}) + (p - 2) \lambda_{\min}\big(D^{2} \phi(x_{0})\big) \leq f\big(x_{0}, \phi(x_0)\big) 
$$
where \( \lambda_{\max}\big(D^{2} \phi(x_{0})\big) \) e \( \lambda_{\min}\big(D^{2} \phi(x_{0})\big) \) are the largest and smallest eigenvalues of the matrix \( D^{2} \phi(x_{0}) \) respectively.
\item For $1<p< 2 $
$$
\Delta \phi(x_{0}) + (p - 2) \lambda_{\min}\big(D^{2} \phi(x_{0})\big) \geq f\big(x_{0}, \phi(x_0)\big)
$$
$$
(\text{resp.}\,\,\,   \Delta \phi(x_{0}) + (p - 2) \lambda_{\max}\big(D^{2} \phi(x_{0})\big) \leq f\big(x_{0}, \phi(x_0)\big)). 
$$
\end{enumerate}
\end{enumerate}

Finally, we say that \( u \) is a solution to \eqref{Vis1} in the viscosity sense if it satisfies the conditions for both subsolutions and supersolutions. 
\end{definition}

The following definition for the case \( \gamma < 0 \) is adapted from the one introduced by Julin and Juutinen in \cite{JJ12} for the singular \( p \)-Laplacian.

\begin{definition}[{\bf \cite[Definition 2.2]{Attou18}}]\label{Def2.2} Let $\Omega \subset \mathbb{R}^n$ be a bounded domain, $1 < p < \infty$, and $-1 < \gamma < 0$. A function $u$ is a viscosity \textit{subsolution} of \eqref{Vis1} if $u \not\equiv \infty$ and if for all $\varphi \in C^2(\Omega)$ such that $u - \varphi$ attains a local maximum at $x_0$ and $\nabla \varphi(x) \neq 0$ for $x \neq x_0$, one has
\[
\lim_{r \to 0} \inf_{x \in B_r(x_0), \atop{x \neq x_0}} \left( -|\nabla \varphi(x)|^\gamma \Delta_p^{\mathrm{N}} \varphi(x) \right) \leq -f(x_0, \varphi(x_0)).
\]

A function $u$ is a viscosity \textit{supersolution} of \eqref{Vis1} if $u \not\equiv \infty$ and for all $\varphi \in C^2(\Omega)$ such that $u - \varphi$ attains a local minimum at $x_0$ and $\nabla \varphi(x) \neq 0$ for $x \neq x_0$, one has
\[
\lim_{r \to 0} \sup_{x \in B_r(x_0), \atop{x \neq x_0}}  \left( -|\nabla \varphi(x)|^\gamma \Delta_p^{\mathrm{N}} \varphi(x) \right)  \geq -f(x_0, \varphi(x_0)).
\]

We say that $u$ is a viscosity solution of \eqref{Vis1} in $\Omega$ if it is both a viscosity subsolution and a viscosity supersolution.
\end{definition}

We emphasize that Birindelli and Demengel introduced an alternative definition of solutions in \cite{BirDem, BirDem06, BirDem07}. This definition represents a variation of the standard notion of viscosity solutions for \eqref{Vis1}, which avoids the use of test functions with vanishing gradients at the testing point.

\begin{definition}[{\bf \cite[Definition 2.3]{Attou18}}]\label{Def2.3}
Let $-1 < \gamma < 0$ and $p > 1$. A lower semicontinuous function $u$ is a viscosity \textit{supersolution} of  \eqref{Vis1} in $\Omega$ if for every $x_0 \in \Omega$ one of the following conditions holds:

\begin{itemize}
    \item[i)] Either for all $\varphi \in C^2(\Omega)$ such that $u - \varphi$ has a local minimum at $x_0$ and $\nabla\varphi(x_0) \neq 0$, we have
    \[
    -|\nabla\varphi(x_0)|^\gamma \Delta_p^{\mathrm{N}} \varphi(x_0) \geq -f(x_0, \varphi(x_0)).
    \]
    \item[ii)] Or there exists an open ball $B(x_0, \delta) \subset \Omega$, $\delta > 0$, such that $u$ is constant in $B(x_0, \delta)$ and $-f(x, u) \leq 0$ for all $x \in B(x_0, \delta)$.
\end{itemize}

An upper semicontinuous function $u$ is a viscosity \textit{subsolution} of  \eqref{Vis1} in $\Omega$ if for all $x_0 \in \Omega$ one of the following conditions holds:
\begin{itemize}
    \item[i)] Either for all $\varphi \in C^2(\Omega)$ such that $u - \varphi$ attains a local maximum at $x_0$ and $\nabla\varphi(x_0) \neq 0$, we have
    \[
   -|\nabla\varphi(x_0)|^\gamma \Delta_p^{\mathrm{N}} \varphi(x_0) \leq -f(x_0, \varphi(x_0)).
    \]
    \item[ii)] Or there exists an open ball $B(x_0, \delta) \subset \Omega$, $\delta > 0$, such that $u$ is constant in $B(x_0, \delta)$ and $-f(x, u) \geq 0$ for all $x \in B(x_0, \delta)$.
\end{itemize}
\end{definition}

The following result ensures that the previous definitions for the singular scenario are equivalent.

\begin{proposition}[{\bf \cite[Proposition 2.4]{Attou18}}] Definitions \ref{Def2.2} and \ref{Def2.3} are equivalent.

\end{proposition}

Next, we observe that the singular case can be reformulated as the analysis of the normalized \( p \)-Laplacian problem incorporating a lower-order term. The proof follows the same ideas as the one in \cite[Lemma 2.5]{Attou18}, and for this reason, we will omit the proof here.

\begin{lemma}\label{Lemma2.5}  Let \( \gamma \in (-1,0) \) and \( p > 1 \). Assume that \( u \) is a viscosity solution to \eqref{Problem}. Then, \( u \) is a viscosity solution to  
\[
\Delta^{\mathrm{N}}_p u = f(x, u) |\nabla u|^{-\gamma} \quad \text{in} \quad \Omega.
\]
\end{lemma}

The following lemma will be instrumental in the subsequent section regarding the proof of the improvement of flatness. It is a sort of cutting Lemma.

\begin{lemma}[{\bf Cutting Lemma, \cite[Lemma 2.6]{Attou18}}]\label{CT-Lemma}  Let \( \gamma > -1 \) and \( p > 1 \). Assume that \( w \) is a viscosity solution of  
\[
- |D w + \vec{\xi}|^{\gamma} \left(\Delta w - (p - 2) \frac{ D^2 w (D w + \vec{\xi}), (D w + \vec{\xi}) }{|D w + \vec{\xi}|^2}\right) = 0.
\]
Then, \( w \) is a viscosity solution of  
\[
-\Delta^{\mathrm{N}}_{p} (w + \vec{\xi}\cdot x)= - \Delta w - (p - 2) \frac{ D^2 w (D w + \vec{\xi}), (D w + \vec{\xi}) }{|D w + \vec{\xi}|^2} = 0.
\]
\end{lemma}

The next result ensures that any viscosity solution to
\begin{equation}\label{Eq-Slop}
- |D w + \vec{\xi}|^{\gamma} \left(\Delta w - (p - 2) \frac{ D^2 w (D w + \vec{\xi}), (D w + \vec{\xi}) }{|D w + \vec{\xi}|^2}\right) = f(x),
\end{equation}
is locally H\"{o}lder continuous provided we are dealing with large slopes.

\begin{lemma}[{\bf \cite[Lemma 3.2]{Attou18}}]\label{Holder-Cont-Large-Slop} Let $\gamma \in (-1, \infty)$ and $p \in (1, \infty)$. Assume that $|\vec{\xi}| \geq 1$ and let $w$ be a viscosity solution to equation \eqref{Eq-Slop}. For all $r \in (0,1)$ and $\beta_{\mathrm{L}} \in (0,1)$, there exists a constant $\mathrm{C} = \mathrm{C}(p, n, \gamma, \beta) > 0$ such that
\[
\displaystyle \sup_{x, y \in B_r \atop{x \neq y}} \frac{|w(x) - w(y)|}{|x-y|^{\beta_{\mathrm{L}}}} \leq \mathrm{C} \left( \|w\|_{L^\infty(B_1)} + \|f\|_{L^\infty(B_1)}^{\frac{1}{1+\gamma}} \right). 
\]
    
\end{lemma}

The next result provides a uniform H\"{o}lder estimate to \eqref{Eq-Slop} for small slopes.

\begin{lemma}[{\bf \cite[Lemma 3.3]{Attou18}}]\label{Holder-Cont-mall-Slop}
For all $r \in (0,1)$, there exist a constant $\beta_{\mathrm{S}} = \beta_{\mathrm{S}}(p, n) \in (0,1)$ and a positive constant $\mathrm{C}_0 = \mathrm{C}_0(p, n, r, \operatorname{osc}_{B_1}(w), \|f\|_{L^n(B_1)})$ such that any viscosity solution $w$ of \eqref{Eq-Slop} with $|\vec{\xi}| \leq 1$ satisfies

\[
[w]_{C^{0,\beta_{\mathrm{S}}}(B_r)} \leq \mathrm{C}_0 .
\]
\end{lemma}

For a small \(\gamma > 0\) and \(p > 1\) close to 2, Attouchi and Ruosteenoja established a local \(W^{2,2}\) estimate. We utilize this estimate to derive an average estimate (in the \(L^2\)-sense) for solutions of \eqref{Problem} along free boundary points.

\begin{theorem}[{\bf \cite[Theorem 1.3]{Attou18}}]\label{Hes-est} Assume that $f \in W^{1,1}(\Omega) \cap L^\infty(\Omega) \cap \mathrm{C}^0(\Omega)$ and $\gamma \in (0, \zeta]$, where $\zeta \in (0,1)$ and 
\[
1 > \zeta + \sqrt{n} |p - 2 - \gamma| + \zeta (p - 2 - \gamma)_+.
\]
Then, any viscosity solution $u$ of 
$$
-|\nabla u|^{\gamma}\Delta_p^{\mathrm{N}} u = f(x) \quad \text{in} \quad \Omega
$$
belongs to $W_{\text{loc}}^{2,2}(\Omega)$, and for any $\Omega'' \subset\subset \Omega' \subset\subset \Omega$, we have

\[
\|u\|_{W^{2,2}(\Omega'')} \leq \mathrm{C}(p,n,\gamma,\zeta, d_{\Omega}, d'', \|u\|_{L^\infty(\Omega)}, \|f\|_{L^\infty(\Omega)}, \|f\|_{W^{1,1}(\Omega)}), \tag{1.4}
\]
where \( d_{\Omega} = \mathrm{diam}(\Omega) \) and \( d'' = \mathrm{dist}(\Omega', \partial \Omega') \).
\end{theorem}

The following result, established by Siltakoski, provides a fundamental connection between viscosity solutions of the normalized \( p \)-Laplacian and weak solutions of the \( p \)-Laplacian.

\begin{theorem}[{\bf \cite[Theorem 5.9]{Siltak18}}]\label{Thm-Equiv-Sol} A function \( u \) is a viscosity solution to  
\[
- \Delta^{\mathrm{N}}_{p} u = 0 \quad \text{in } \quad \Omega
\]  
if and only if it is a weak solution to  
\[
- \Delta_{p} u = 0 \quad \text{in } \quad \Omega.
\]
    
\end{theorem}

The next result addresses gradient estimates for quasilinear models with bounded forcing terms.

\begin{theorem}[{\bf \cite[Theorem 1.1]{Attou18} and \cite[Theorem 1.1]{WYF2025}}]\label{Thm-grad-est}
Let \( \gamma > -1 \), \( p > 1 \), and \( f \in L^\infty(\Omega) \cap C^{0}(\Omega) \) be H\"{o}lder continuous when \( \gamma < 0 \). Then, there exists \( \alpha = \alpha(p, n, \gamma) > 0 \) such that any viscosity solution \( u \) of $|\nabla u|^{\gamma} \Delta_p^{\mathrm{N}} u = f(x)$ in $\Omega$ belongs to \( C^{1, \alpha}_{\text{loc}}(\Omega) \), and for any \( \Omega' \subset\subset \Omega \),
\[
[u]_{C^{1, \alpha}(\Omega')} \leq \mathrm{C} = \mathrm{C}(p, n, d
_{\Omega}, \gamma, d', \|u\|_{L^\infty(\Omega)}, \|f\|_{L^\infty(\Omega)}),
\]
where \( d_{\Omega} = \mathrm{diam}(\Omega) \) and \( d' = \mathrm{dist}(\Omega', \partial \Omega) \).

Particularly, we have the following gradient estimate:
$$
\|\nabla u\|_{L^\infty(\Omega')} \leq \mathrm{C}(p, n, d
_{\Omega}, \gamma, d')\left( \|u\|_{L^\infty(\Omega)} + \|f\|^{\frac{1}{\gamma+1}}_{L^\infty(\Omega)}\right).
$$
\end{theorem}
\begin{remark}
 In {\bf \cite[Theorem 1.1]{Attou18}}, after the proof of the theorem, there exists a note where the authors show that in the case $-1<\gamma<0$, the Hölder regularity assumption of $f$ can be removed. 
\end{remark}
\subsection*{Existence/Uniqueness of viscosity solutions}

To formulate our existence and uniqueness result in a general framework, we introduce a Hamiltonian term. Specifically, let \( \mathscr{H}: \Omega \times \mathbb{R}^n \to \mathbb{R} \) be a continuous function satisfying the following properties:
\begin{itemize}
    \item[(\textbf{H1})] \( \mathscr{H}(x, \vec{\xi}) \leq \mathrm{C}(1 + |\vec{\xi}|^\kappa) \) for some \( \kappa \in (0, 1 + \gamma] \) and all \( (\vec{\xi}, x) \in \mathbb{R}^n \times \Omega \) (for \(\gamma > -1\));
    \item[(\textbf{H2})] \( |\mathscr{H}(x, \vec{\xi}) - \mathscr{H}(y, \vec{\xi})| \leq \omega(|x - y|)(1 + |\vec{\xi}|^\kappa) \), where \( \omega: [0, \infty) \to [0, \infty) \) is a modulus of continuity, i.e., an increasing function with \( \omega(0) = 0 \).
\end{itemize}

Therefore, we study viscosity solutions to equations of the form
\[
|\nabla v|^{\gamma} \Delta_p^{\mathrm{N}} v + \mathscr{H}(x, \nabla v) + f_0(x, v) = 0 \quad \text{in} \quad \Omega, \quad \text{and} \quad v = g \quad \text{on} \quad \partial \Omega, \tag{2.1}
\]
where \( f_0 \) and \( g \) are assumed to be continuous.

The archetypal model we consider is
\[
\mathcal{Q}^{\mathrm{N}}_{p, \gamma} v \coloneqq |\nabla v|^{\gamma} \Delta_p^{\mathrm{N}} v + |\nabla v|^{\gamma} \vec{\mathfrak{B}}(x) \cdot \nabla v = \mathfrak{a}(x) v_+^m(x) + \mathfrak{h}(x),
\]
where we assume conditions \((\mathrm{H}0)\)--\((\mathrm{H}2)\) and \(\mathrm{A}\ref{Assumption_lambda}\), with \(\mathfrak{h} \in \mathrm{C}^0(\Omega)\) and \(\vec{\mathfrak{B}} \in \mathrm{C}^0(\overline{\Omega}; \mathbb{R}^n)\).

For our next result, we will require the notion of superjets and subjets, as introduced by Crandall, Ishii, and Lions in \cite{CIL}.

\begin{definition}
    A second-order superjet of \( u \) at \( x_0 \in \Omega \) is defined as
\[
\mathcal{J}^{2,+}_{\Omega} u(x_0) = \left\{ (\nabla \phi(x_0), D^2 \phi(x_0)) : \phi \in C^2 \text{ and } u - \phi \text{ attains a local maximum at } x_0 \right\}.
\]
The closure of a superjet is given by
\begin{eqnarray*}
\overline{\mathcal{J}}^{2,+}_{\Omega} u(x_0) = \left\{ (\vec{\xi}, \mathrm{X}) \in \mathbb{R}^n \times \text{Sym}(n): \exists\,\, (\vec{\xi}_k, \mathrm{X}_k) \in \mathcal{J}^{2,+}_{\Omega} u(x_k) \text{ such that}\right.\\ 
\left.(x_k, u(x_k), \vec{\xi}_k, \mathrm{X}_k) \to (x_0, u(x_0), \vec{\xi}, \mathrm{X}) \right\}.
\end{eqnarray*}
Similarly, we can define the second-order subjet and its closure.
\end{definition}

The following result is fundamental for establishing the existence of viscosity solutions to our problem and for deriving certain weak geometric properties in subsequent sections. It represents a more generalized formulation, which will prove instrumental for our objectives in future research.

\begin{lemma}[\bf Comparison Principle]\label{Comp-Princ} Let $\mathfrak{c}, \mathfrak{h}
_1, \mathfrak{h}_2\in \mathrm{C}^0(\overline{\Omega})$, and let $\mathfrak{F}:\R \to \R$ be a continuous and increasing function satisfying $\mathfrak{F}(0)=0$. Suppose that 
\begin{equation}\label{CP}
\left\{
\begin{array}{cccccccc}
|\nabla u_1|^{\gamma} \Delta_p^{\mathrm{N}} u_1 +\mathscr{H}(x,\nabla u_1)+ \mathfrak{c}(x) \mathfrak{F}(u_1) & \geq & \mathfrak{h}_1(x) &\text{in}& \Omega,\\[0.2cm]
|\nabla u_2|^{\gamma} \Delta_p^{\mathrm{N}} u_2 +\mathscr{H}(x,\nabla u_2) +  \mathfrak{c}(x) \mathfrak{F}(u_2) & \leq & \mathfrak{h}_2(x)  &\text{in} & \Omega,
    \end{array}
\right.
\end{equation}
in the viscosity sense. Furthermore, assume that  $u_2 \geq u_1$ on $\partial \Omega$ and one of the following holds:
\begin{enumerate}
\item[$\textbf{(a)}$] If $\mathfrak{c}<0$ in $\overline{\Omega}$ and $h_1\geq h_2$ in $\overline{\Omega}$,
\item[$\textbf{(b)}$] If $\mathfrak{c}\leq 0$ in $\overline{\Omega}$ and $h_1>h_2$ in $\overline{\Omega}$.
\end{enumerate}
Then, $u_2 \geq u_1$ in $\Omega$.
\end{lemma}

\begin{proof}
Assume, for the sake of contradiction, that there exists a constant $\mathrm{M}_0 > 0$ such that
\[
\mathrm{M}_0 \coloneqq \sup_{\overline{\Omega}} (u_1 - u_2) > 0.
\]
For each $\varepsilon > 0$, define
\begin{equation}\label{maximum}
\mathrm{M}_\varepsilon = \sup_{\overline{\Omega} \times \overline{\Omega}} \left( u_1(x) - u_2(y) - \frac{1}{2\varepsilon} |x - y|^2 \right) < \infty.
\end{equation}
Let $(x_\varepsilon, y_\varepsilon) \in \overline{\Omega} \times \overline{\Omega}$ denote the points where the supremum is achieved. Following the argument in \cite[Lemma 3.1]{CIL}, we have
\begin{equation}\label{CP1}
\lim_{\varepsilon \to 0} \frac{1}{\varepsilon} |x_\varepsilon - y_\varepsilon|^2 = 0 \quad \text{and} \quad \lim_{\varepsilon \to 0} \mathrm{M}_\varepsilon = \mathrm{M}_0.
\end{equation}
In particular,
\begin{equation}\label{CP2}
z_0 \coloneqq \lim_{\varepsilon \to 0} x_\varepsilon = \lim_{\varepsilon \to 0} y_\varepsilon,
\end{equation}
with $u_1(z_0) - u_2(z_0) = \mathrm{M}_0$. Since
\[
\mathrm{M}_0 > 0 \geq \sup_{\partial \Omega} (u_1 - u_2),
\]
it follows that $x_\varepsilon, y_\varepsilon \in \Omega'$ for some interior domain $\Omega' \Subset \Omega$ and for all sufficiently small $\varepsilon > 0$. By \cite[Theorem 3.2]{CIL}, there exist matrices $\mathrm{X}, \mathrm{Y} \in \text{Sym}(n)$ such that
\begin{equation}\label{CP3}
\left( \frac{x_\varepsilon - y_\varepsilon}{\varepsilon}, \mathrm{X} \right) \in \overline{\mathcal{J}}^{2,+} u_1(x_\varepsilon) \quad \text{and} \quad \left( \frac{y_\varepsilon - x_\varepsilon}{\varepsilon}, \mathrm{Y} \right) \in \overline{\mathcal{J}}^{2,-} u_2(y_\varepsilon),
\end{equation}
and
\begin{equation}\label{CP4}
-\frac{3}{\varepsilon} \begin{pmatrix}
\mathrm{Id}_n & 0 \\
0 & \mathrm{Id}_n
\end{pmatrix} \leq \begin{pmatrix}
\mathrm{X} & 0 \\
0 & -\mathrm{Y}
\end{pmatrix} \leq \frac{3}{\varepsilon} \begin{pmatrix}
\mathrm{Id}_n & -\mathrm{Id}_n \\
-\mathrm{Id}_n & \mathrm{Id}_n
\end{pmatrix}.
\end{equation}
In particular, $\mathrm{X} \leq \mathrm{Y}$. Letting $\eta_\varepsilon = \frac{x_\varepsilon - y_\varepsilon}{\varepsilon}$ and using \eqref{CP} and \eqref{CP3}, we derive
\begin{align*}
\mathfrak{h}_1(x_\varepsilon) &\leq |\eta_\varepsilon|^\gamma \left( \text{Tr}(\mathrm{X}) + (p - 2) \left\langle \mathrm{X} \eta_\varepsilon, \eta_\varepsilon \right\rangle \right) + \mathfrak{c}(x_\varepsilon) \mathfrak{F}(u_1(x_\varepsilon)) + \mathscr{H}(x_\varepsilon, \eta_\varepsilon) \\
&\leq |\eta_\varepsilon|^\gamma \left( \text{Tr}(\mathrm{Y}) + (p - 2) \left\langle \mathrm{Y} \eta_\varepsilon, \eta_\varepsilon \right\rangle \right) + \mathfrak{c}(x_\varepsilon) \mathfrak{F}(u_1(x_\varepsilon)) + \mathscr{H}(x_\varepsilon, \eta_\varepsilon) \\
&\leq \mathfrak{h}_2(y_\varepsilon) - \mathfrak{c}(y_\varepsilon) \mathfrak{F}(u_2(y_\varepsilon)) - \mathscr{H}(x_\varepsilon, \eta_\varepsilon) + \mathfrak{c}(x_\varepsilon) \mathfrak{F}(u_1(x_\varepsilon)) + \mathscr{H}(x_\varepsilon, \eta_\varepsilon) \\
&\leq \mathfrak{h}_2(y_\varepsilon) + \mathfrak{F}(u_2(y_\varepsilon)) (\mathfrak{c}(x_\varepsilon) - \mathfrak{c}(y_\varepsilon)) + \left[ \min_{\overline{\Omega}} \mathfrak{c} \right] (\mathfrak{F}(u_1(x_\varepsilon)) - \mathfrak{F}(u_2(y_\varepsilon))) \\
&\quad + \omega(|x_\varepsilon - y_\varepsilon|) (1 + |\eta_\varepsilon|^\kappa).
\end{align*}
Taking the limit as $\varepsilon \to 0$ in the above inequality and applying \eqref{CP1} and \eqref{CP2}, we obtain
\[
\mathfrak{h}_1(z_0) - \mathfrak{h}_2(z_0) \leq \left[ \min_{\overline{\Omega}} \mathfrak{c} \right] (\mathfrak{F}(u_1(z_0)) - \mathfrak{F}(u_2(z_0))).
\]
This contradicts assumptions $\textbf{(a)}$ and $\textbf{(b)}$. Therefore, we conclude that $u_2 \geq u_1$ in $\Omega$.
\end{proof}

\bigskip

We now address the existence of a viscosity solution to the Dirichlet problem \eqref{Problem} (respectively, \eqref{Bound-Cond}). This result is derived by employing Perron's method, provided that a suitable version of the Comparison Principle holds. Specifically, consider the functions $u^{\sharp}$ and $u_{\flat}$, which satisfy the following boundary value problems:
\begin{equation}
\left\{
\begin{array}{rclcl}
|\nabla u^{\sharp}|^{\gamma} \Delta_p^{\mathrm{N}} u^{\sharp} & = & 0 & \text{in } &\Omega, \\
u^{\sharp}(x) & = & g(x) & \text{on }  & \partial \Omega,
\end{array}
\right.
\end{equation}
and
\begin{equation}
 \left\{
\begin{array}{rclcl}
|\nabla u_{\flat}|^{\gamma} \Delta_p^{\mathrm{N}} u_{\flat} & = & \|g\|_{L^{\infty}(\Omega)}^{m} & \text{in } &\Omega, \\
u_{\flat}(x) & = & g(x) & \text{on }  & \partial \Omega.
\end{array}
\right.
\end{equation}

The existence of such solutions can be established using standard techniques. Moreover, it is evident that $u^{\sharp}$ and $u_{\flat}$ act as a supersolution and a subsolution, respectively, to \eqref{Problem}. Therefore, by invoking the Comparison Principle (Lemma \ref{Comp-Princ}), Perron's method guarantees the existence of a viscosity solution in $\mathrm{C}^0(\Omega)$ to \eqref{Problem}. More precisely, we obtain the following theorem.

\begin{theorem}[\bf Existence and Uniqueness]\label{existence_uniquiness}
Let $f \in \mathrm{C}^0([0,\infty))$ be a bounded, increasing real-valued function satisfying $f(0) = 0$. Suppose that there exist a viscosity subsolution $u_{\flat} \in \mathrm{C}^0(\Omega) \cap C^{0,1}(\Omega)$ and a viscosity supersolution $u^{\sharp} \in \mathrm{C}^0(\Omega) \cap C^{0,1}(\Omega)$ to the equation
\begin{equation}
    |\nabla u|^{\gamma} \Delta_p^{\mathrm{N}} u = \mathfrak{a}(x)f(u) \quad \text{in} \quad \Omega,
\end{equation}
such that $u_{\flat} = u^{\sharp} = g \in \mathrm{C}^0(\partial \Omega)$. Define the class of functions
\begin{equation}
\mathcal{S}_g(\Omega) := \left\{ v \in \mathrm{C}^0(\Omega) \mid v \text{ is a viscosity supersolution to } |\nabla u|^{\gamma} \Delta_p^{\mathrm{N}} u = \mathfrak{a}(x)f(u) \text{ in } \Omega, \right. 
\end{equation}
\begin{equation*}
    \left. \text{such that } u_{\flat} \leq v \leq u^{\sharp} \text{ and } v = g \text{ on } \partial \Omega \right\}.
\end{equation*}
Then, the function
\begin{equation}
    u(x) := \inf_{\mathcal{S}_g(\Omega)} v(x), \quad \text{for } x \in \Omega,
\end{equation}
is the unique continuous (up to the boundary) viscosity solution to the problem
\begin{equation}
\left\{
\begin{array}{rclcl}
|\nabla u|^{\gamma} \Delta_p^{\mathrm{N}} u & = & \mathfrak{a}(x)f(u) & \text{in } & \Omega, \\
u(x) & = & g(x) & \text{on }& \partial \Omega.
\end{array}
\right.
\end{equation}
\end{theorem}

\section{Geometric regularity estimates}

\subsection{Regularity across the free boundary}

We begin by making a few observations regarding the scaling properties of the model, which will prove useful throughout the paper. Suppose that $u$ is a viscosity solution to \eqref{Problem}. Let $\kappa$ and $r$ be positive constants, and define
\[
v(x) = \frac{u(rx)}{\kappa}.
\]
A straightforward calculation reveals that $v$ is a viscosity solution to
\[
|\nabla v(x)|^\gamma \Delta_p^\mathrm{N} v(x) = \mathfrak{a}_{\kappa,r}(x) v_+^m(x),
\]
where $\mathfrak{a}_{\kappa,r}(x) \coloneqq \left(r^{2+\gamma} / \kappa^{\gamma+1-m}\right) \mathfrak{a}(rx)$. By choosing $\kappa = \|u\|_{L^\infty(B_1)}$, we may assume, without loss of generality, that solutions are normalized, i.e., $\|u\|_{L^\infty(B_1)} \leq 1$.

\begin{lemma}[{\bf Flatness Estimate}]\label{Flatness Estimate}
Let $\delta>0$ and suppose that $\mathrm{A} {\ref{Assumption_lambda}}$ holds. Then, there exists $r\coloneqq r(n,\delta)\in (0,1)$ such that if  $u\in \mathrm{C}^0(B_1)$ is a normalized viscosity solution to
\[
\left\{
\begin{array}{rllll}
  |\nabla u(x)|^{\gamma}\NDelta u(x) &=& \zeta^2\mathfrak{a}(x)u^{m}_+(x) & \mbox{in} &B_1, \\
    u(0) &=&0  ,
\end{array}
\right.
\]
where $\zeta\in (0,r]$, we obtain 
\begin{equation}\label{flatness}
\sup_{B_{1/2}}u\leq\delta.
\end{equation}
\end{lemma}

\begin{proof}
Assume, for the sake of contradiction, that the conclusion of the lemma does not hold. Then, there exist $\delta_0 > 0$ and a sequence of normalized non-negative functions $(u_k)$ satisfying
\begin{equation}\label{eq_seq1}
\left\{
\begin{array}{rclccc}
|\nabla u_k(x)|^{\gamma} \NDelta u_k(x) &=& \left(\frac{1}{k}\right)^2 \mathfrak{a}(x) (u_{k})_+^{m}(x) & \text{in} & B_1, \\[0.2cm]
u_k(0) &=& 0.
\end{array}
\right.
\end{equation}
However,
\begin{equation}\label{contra1}
\sup_{B_{1/2}} u_k > \delta_0,
\end{equation}
for every $k \geq 1$. Since the right-hand side of the equation \eqref{eq_seq1} belongs to $L^{\infty}(B_1) \cap \mathrm{C}^0(B_1)$, it follows from Theorem \ref{Thm-grad-est} that for every $k \geq 1$,
\[
\|u_k\|_{C^{1,\alpha}(B_{9/10})} \leq \mathrm{C},
\]
where $\mathrm{C} > 0$ is a constant independent of $k$. Consequently, there exists a function $u_{\infty} \in \mathrm{C}^{0,\alpha}(B_{8/9})$, for some $\alpha \in (0,1)$, satisfying $0 \leq u_{\infty} \leq 1$ and $u_{\infty}(0) = 0$, such that $u_k \to u_{\infty}$ locally uniformly in $B_{8/9}$. By a slight adaptation of \cite[Appendix]{Attou18}, the structural stability of the problem \eqref{eq_seq1} implies that $u_{\infty}$ solves the limit problem in the viscosity sense. That is, $u_{\infty}$ is a viscosity solution of
\begin{equation*}
\left\{
\begin{array}{rclccc}
|\nabla u_{\infty}(x)|^{\gamma} \NDelta u_{\infty}(x) &=& 0 & \text{in} & B_{8/9}, \\[0.2cm]
u_{\infty}(0) &=& 0.
\end{array}
\right.
\end{equation*}
By Lemma \ref{CT-Lemma}, the function $u_{\infty}$ is also a viscosity solution to
\[
\left\{
\begin{array}{rcl}
\NDelta u_{\infty}(x) & = & 0 \quad \text{in} \quad B_{8/9}, \\[0.2cm]
u_{\infty}(0) & = & 0.
\end{array}
\right.
\]
Finally, by Theorem \ref{Thm-Equiv-Sol} and the Strong Maximum Principle from \cite{Vazq84}, we conclude that $u_{\infty} \equiv 0$ in $B_{8/9}$. This contradicts \eqref{contra1} for sufficiently large $k$.
\end{proof}

We are now in a position to present the proof of the improved regularity estimate.

\begin{proof}[{\bf Proof of Theorem \ref{Improved Regularity}}]
Assume without loss of generality that $x_0 = \vec{0}$. Define an auxiliary function
\[
\varphi_1(x) \coloneqq \kappa u(\tau x),
\]
where $\kappa \coloneqq 1/\|u\|_{L^{\infty}(B_1)}$ and $\tau \in (0,1)$ is a constant to be determined later. Observe that
\begin{align*}
|\nabla \varphi_1(x)|^{\gamma} \NDelta \varphi_1(x) &= \kappa^{\gamma} \tau^{\gamma} |\nabla u(x)|^{\gamma} \left[ \kappa \tau^2 \Delta u(\tau x) + (p-2) \kappa \tau^2 \Delta_{\infty}^N u(\tau x) \right] \\
&= \kappa^{1+\gamma} \tau^{2+\gamma} |\nabla u(\tau x)|^{\gamma} \NDelta u(\tau x),
\end{align*}
which implies that $\varphi_1(x)$ is a viscosity solution of the problem
\[
|\nabla \varphi_1(x)|^{\gamma} \NDelta \varphi_1(x) = \kappa^{1+\gamma-m} \tau^{2+\gamma} \mathfrak{a}_1(x) \varphi_1^{m}(x) \quad \text{in} \quad B_1,
\]
where $\mathfrak{a}_1(x) \coloneqq \mathfrak{a}(\tau x)$. The previous lemma ensures the existence of $r \coloneqq r(n, \delta) \in (0,1)$ for a given
\begin{equation}\label{escolha do delta}
\delta \coloneqq 2^{-\beta} > 0,
\end{equation}
such that for any function $\psi$ satisfying $0 \leq \psi \leq 1$, $\psi(0) = 0$, and solving
\[
|\nabla \psi(x)|^{\gamma} \NDelta \psi(x) = \zeta^2 \mathfrak{a}(x) \psi^{m}_+(x) \quad \text{in} \quad B_1,
\]
in the viscosity sense, where $0 < \zeta \leq r$, it holds that
\[
\sup_{B_{1/2}} \psi \leq \delta.
\]
Fix this value of $r(n, \delta)$ for the given $\delta > 0$ in \eqref{escolha do delta}, and with this fixed value, select the constant
\[
\tau \coloneqq r \kappa^{-\frac{1}{\beta}}.
\]
Thus, the auxiliary function $\varphi_1$ satisfies
\[
\begin{cases}
|\nabla \varphi_1(x)|^{\gamma} \NDelta \varphi_1(x) = \zeta^2 \mathfrak{a}_1(x) \varphi_1^m(x) \quad \text{in} \quad B_1, \\
\varphi_1(0) = 0, \\
0 \leq \varphi_1 \leq 1,
\end{cases}
\]
where $\zeta \leq r$. Applying Lemma \ref{flatness} to $\varphi_1$, we obtain
\[
\sup_{B_{1/2}} \varphi_1 \leq 2^{-\beta}.
\]
Next, define a second auxiliary function by
\[
\varphi_2(x) = 2^{\beta} \varphi_1\left(\frac{x}{2}\right).
\]
Observe that
\begin{align*}
|\nabla \varphi_2(x)|^{\gamma} \NDelta \varphi_2(x) 
&= 2^{\frac{(2+\gamma)(1+\gamma)}{1+\gamma-m}} \cdot 2^{-(2+\gamma)} |\nabla \varphi_1\left(\frac{x}{2}\right)|^{\gamma} \NDelta \varphi_1\left(\frac{x}{2}\right) \\
&= 2^{\frac{(2+\gamma)(1+\gamma) - (2+\gamma)(1+\gamma-m)}{1+\gamma-m}} \zeta^2 \mathfrak{a}_1\left(\frac{x}{2}\right) \varphi_1^m\left(\frac{x}{2}\right) \\
&= 2^{\frac{(2+\gamma)m}{1+\gamma-m}} \zeta^2 \cdot 2^{-\frac{(2+\gamma)m}{1+\gamma-m}} \mathfrak{a}_2(x) \varphi_2^m(x) \\
&= \zeta^2 \mathfrak{a}_2(x) \varphi_2^m(x),
\end{align*}
where $\mathfrak{a}_2(x) \coloneqq \mathfrak{a}_1(x/2) = \mathfrak{a}\left((\tau/2)x\right)$. It follows that $\varphi_2$ satisfies the hypotheses of Lemma \ref{flatness}, and consequently,
\[
\sup_{B_{1/2}} \varphi_2 \leq 2^{-\beta}.
\]
From the definition of $\varphi_2$, we obtain
\[
\sup_{B_{1/4}} \varphi_1 \leq 2^{-2 \cdot \beta}.
\]
Repeating this process for the function
\[
\varphi_j(x) \coloneqq 2^{\beta} \varphi_1\left(\frac{x}{2^j}\right),
\]
we derive the estimate
\[
\sup_{B_{1/2^j}} \varphi_j \leq 2^{-j \cdot \beta}.
\]
Finally, for a given $0 < \rho \leq r/2$, let $j \geq 1$ be an integer such that $2^{-(j+1)} \leq \rho/r \leq 2^{-j}$. Then,
\begin{align*}
\kappa \cdot \sup_{B_{\rho}} u(x) 
&\leq \sup_{B_{\rho/r}} \varphi_1(x) \\
&\leq \sup_{B_{2^{-k}}} \varphi_1(x) \\
&\leq 2^{-j \cdot \beta} \\
&= 2^{\beta} \cdot 2^{-(k+1)\beta} \\
&\leq 2^{\beta} \left(\frac{\rho}{r}\right)^{\beta} \\
&\leq \left(\frac{2}{r}\right)^{\beta} \rho^{\beta}.
\end{align*}
Therefore,
\begin{equation}\label{eq_improved}
\sup_{B_{\rho}} u(x) \leq \left(\frac{2}{r}\right)^{\beta} \|u\|_{L^{\infty}(B_1)} \rho^{\beta}.
\end{equation}
\end{proof}

\begin{remark}\label{remark of improved regularity}
An examination of the proof reveals that we only utilize the fact that $\mathfrak{a}$ is bounded from above. This observation allows us to derive the same regularity estimate even when the bounded Thiele modulus depends on $u$. More generally, the result holds when the right-hand side is a function $f(x, u) \in L^{\infty}(B_1\times [0, \mathfrak{L}])$ for which there exists a constant $\mathrm{C}_0 > 0$ satisfying:
\[
f(x, \mu t) \leq \mathrm{C}_0 \mu^{m} f(x, t) \quad \text{for all} \,\,\, (x, t) \in \Omega \times [0, \mathfrak{L}_0],
\]
for some sufficiently small $\mu > 0$. Consequently, the constant appearing in \eqref{eq_improved} also depends on $\|f\|_{L^{\infty}(B_1 \times [0, \mathfrak{L}])}$.
\end{remark}

Now, we will prove Corollary \ref{Cor1.4}.

\begin{proof}[{\bf Proof of Corollary \ref{Cor1.4}}]
Fixed \( x_0 \in \{ u > 0 \} \cap \Omega' \), denote \( \mathrm{d} := \mathrm{dist} \left( x_0, \partial \{ u > 0 \} \right) \). Now, choose \( z_0 \in \partial \{ u > 0 \} \), a free boundary point that achieves such a distance, i.e., \( \mathrm{d} = |x_0 - z_0| \). Hence, from Theorem \ref{Improved Regularity}, we have  
\[
u(x_0) \leq \sup_{B_{\mathrm{d}}(x_0)} u(x) \leq \sup_{B_{2\mathrm{d}}(z_0)} u(x) \leq \mathrm{C}_{\sharp} \left( n, p, \gamma, m, \Lambda_0 \right) \mathrm{d}^{\beta},
\]
which finishes the proof.
\end{proof}

As a significant application of Theorem \ref{Improved Regularity}, we obtain a more precise gradient estimate for solutions of \eqref{Problem} near their free boundary points (cf. \cite[Lemma 3.8]{daSRS19}).

\begin{proof}[{\bf Proof of Theorem \ref{Thm-Sharp-Grad}}] Let \( x_0 \in \partial \{ u > 0 \} \cap \Omega' \) be an interior free boundary point. Now, we define the scaled auxiliary function \( \Phi \colon B_1 \to \mathbb{R}^+ \) given by
\[
\Phi(x) := \frac{u(x_0 + r x)}{r^{\beta}}.
\]
Observe that \( \Phi \) fulfills in the viscosity sense:
\[
|\nabla\Phi|^{\gamma} \Delta_p^{\mathrm{N}} \Phi = \tilde{\mathfrak{a}}(x) \Phi^m(x) \quad \text{in} \quad B_1,
\]
where $\tilde{\mathfrak{a}}(x) := \mathfrak{a}(z + r x)$.

Therefore, from Theorem \eqref{Improved Regularity}, we obtain
\[
\sup_{B_1} \Phi(x) \leq \mathrm{C}^{\prime}(\text{universal}).
\]
Finally, by invoking the available gradient estimates from Theorem \ref{Thm-grad-est}, we obtain 
$$
\begin{array}{rcl}
\displaystyle \frac{1}{r^{\beta}} \displaystyle \sup_{B_{r/2}(x_0)} |\nabla u(x)| & = & \displaystyle \sup_{B_{1/2}(x_0)} |\nabla \Phi(y)| \\
&\leq & \displaystyle \mathrm{C}(n, p, \gamma)
\left[ \sup_{B_1} \Phi(x) + \left( \sup_{B_1} \Phi(x) \right)^{\frac{m}{\gamma+1}} \right]\\
&\leq & \mathrm{C}(n, p, \gamma).\max\left\{ \mathrm{C}^{\prime},  \left(\mathrm{C}^{\prime}\right)^{\frac{m}{\gamma+1}}\right\}
\end{array}
$$
which finishes the proof.

\end{proof}

\subsection{Non-degeneracy and measure-theoretic estimates}

Next, we present the proof of the non-degeneracy of solutions by constructing a suitable barrier function.

\begin{proof}[{\bf Proof of Theorem \ref{Thm-Non-Deg}}]
Assume without loss of generality that $y = \vec{0}$. Since $u \in \mathrm{C}^0(\Omega)$, it suffices to prove the theorem for points $y \in \{u > 0\} \cap B_{1/2}$. Thus, we may assume that $u$ is strictly positive. Define the auxiliary function
\[
\Psi(x) \coloneqq \mathrm{c} |x|^{\beta},
\]
where $\beta \coloneqq (2 + \gamma)/(1 + \gamma - m)$ and $\mathrm{c} = \mathrm{C}_{\mathrm{ND}}(n, \gamma, p, m, \lambda_0)$ is a small positive constant to be determined later. Observe that
\[
\nabla \Psi(x) = \mathrm{c} \beta |x|^{\beta - 2} x \quad \text{and} \quad D^{2} \Psi(x) = \mathrm{c} \beta |x|^{\beta - 2} \left[ (\beta - 2) \frac{x \otimes x}{|x|^2} + \mathrm{Id}_n \right].
\]
Consequently, we have
\begin{align*}
|\nabla \Psi(x)|^{\gamma} \NDelta \Psi(x)
&= |\nabla \Psi(x)|^{\gamma} \left[ \Delta \Psi(x) + (p - 2) \Delta_{\infty}^N \Psi(x) \right] \\
&= \mathrm{c}^{\gamma} \beta^{\gamma} \left[ \mathrm{c} \beta (\beta - 2 + n) |x|^{\beta - 2} + (p - 2) \mathrm{c} \beta (\beta - 1) |x|^{\beta - 2} \right] |x|^{(\beta - 1)\gamma} \\
&= \mathrm{c}^{1 + \gamma} \beta^{1 + \gamma} [n - 1 + (p - 1)(\beta - 1)] |x|^{(\beta - 1)\gamma + \beta - 2} \\
&= \mathrm{c}^{1 + \gamma} \beta^{1 + \gamma} [n - 1 + (p - 1)(\beta - 1)] |x|^{m \beta}.
\end{align*}
To select the appropriate value of $\mathrm{c}$, it is sufficient to choose
\[
0 < \mathrm{c} \leq \left[ \frac{(1 + \gamma - m)^{2 + \gamma} \lambda_0}{(2 + \gamma)^{1 + \gamma} [(1 + \gamma - m)(n - 1) + (p - 1)(1 + m)]} \right]^{\frac{1}{1 + \gamma - m}},
\]
where $\lambda_0$ is a positive constant given by assumption $\mathrm{A}{\ref{Assumption_lambda}}$. Thus, we can estimate
\[
|\nabla \Psi(x)|^{\gamma} \NDelta \Psi(x) - \mathfrak{a}(x) \Psi^m_+(x) \leq 0 = |\nabla u(x)|^{\gamma} \NDelta u(x) - \mathfrak{a}(x) u^m_+(x) \quad \text{in} \quad B_1.
\]
Now, suppose that for some $x_0 \in \partial B_r$, with $0 < r < 1/2$, we have
\begin{equation}\label{u>g}
u(x_0) > \Psi(x_0).
\end{equation}
It follows that
\[
\sup_{\overline{B}_r} u \geq u(x_0) \geq \Psi(x_0) = \mathrm{c} \cdot r^{\beta}.
\]
Therefore, it suffices to verify whether \eqref{u>g} holds. Suppose, for contradiction, that it does not hold, i.e., $u(x) < \Psi(x)$ for all $x \in \partial B_r$. Since $u$ and $\Psi$ satisfy
\[
\left\{
\begin{array}{rclcl}
|\nabla \Psi(x)|^{\gamma} \NDelta \Psi(x) - \Psi^m_+(x) & \leq & |\nabla u(x)|^{\gamma} \NDelta u(x) - u^m_+(x) & \text{in} & B_r, \\
\Psi(x) & > & u(x) & \text{on} & \partial B_r,
\end{array}
\right.
\]
the Comparison Principle (Lemma \ref{Comp-Princ}) implies
\[
\Psi(x) \geq u(x) \quad \text{in} \quad B_r.
\]
In particular, at the point $y = 0$, we have $\Psi(0) = 0$, which implies
\[
0 = \Psi(0) \geq u(0) > 0,
\]
a contradiction. Thus, \eqref{u>g} must hold, and the proof is complete.
\end{proof}

Next, we will deliver some measure-theoretic properties.

\begin{proof}[{\bf Proof of Corollary \ref{Corollary-UPD}}]
Let $r \in (0, 1/2)$ be fixed. From Theorem \ref{Thm-Non-Deg}, there exists a point $y_0 \in \overline{B}_r(x_0)$ such that
\[
u(y_0) = \sup_{\overline{B}_r(x_0)} u \geq \mathrm{C}_{\mathrm{ND}} \cdot r^{\beta}.
\]
On the other hand, the following inclusion holds:
\begin{equation}\label{tau_r}
B_{\tau r}(y_0) \subset \{u > 0\},
\end{equation}
for some sufficiently small $\tau > 0$. To verify this, suppose, for contradiction, that
\[
\partial\{u > 0\} \cap B_{\tau r}(y_0) \neq \emptyset.
\]
Let $z_0 \in \partial\{u > 0\} \cap B_{\tau r}(y_0)$. By Theorem \ref{Improved Regularity}, we have
\[
u(y_0) \leq \mathrm{C}_{\mathcal{G}} |y_0 - z_0|^{\beta}.
\]
Consequently,
\[
\mathrm{C}_{\mathrm{ND}} r^{\beta} \leq u(y_0) \leq \mathrm{C}_{\mathcal{G}} |y_0 - z_0|^{\beta} \leq \mathrm{C}_{\mathcal{G}} (\tau r)^{\beta}.
\]
This leads to a contradiction if we choose
\begin{equation}\label{tau_constant}
0 < \tau < \left(\frac{\mathrm{C}_{\mathrm{ND}}}{\mathrm{C}_{\mathcal{G}}}\right)^{\frac{1}{\beta}}.
\end{equation}
Thus, for such $\tau$, the inclusion \eqref{tau_r} holds. 

Finally, we conclude that
\[
\mathscr{L}^n(B_r(x_0) \cap \{u > 0\}) \geq \mathscr{L}^n(B_r(x_0) \cap B_{\tau r}(y_0)) \geq \theta r^n,
\]
where $\theta > 0$ is a constant depending on $\tau$ and the dimension $n$.
\end{proof}

Next, we present the proof of Corollary \ref{Corollary-Non-Deg}.

\begin{proof}[{\bf Proof of Corollary \ref{Corollary-Non-Deg}}]
Assume, for the sake of contradiction, that such a constant \( \mathrm{C}_{\ast} > 0 \) does not exist. Then, there exists a sequence \( x_k \in \{ u > 0 \} \cap \Omega \) with  
\( d_k := \mathrm{dist}(x_k, \partial \{ u > 0 \} \cap \Omega) \to 0 \) as \( k \to \infty \), and  
\[
u(x_k) \leq k^{-1} d_k^{\beta}.
\]

Define the auxiliary function \( v_k : B_1 \to \mathbb{R} \) by  
\[
v_k(z) := \frac{u(x_k + d_k z)}{d_k^{\beta}}.
\]
It is straightforward to verify the following properties:
\begin{enumerate}
    \item \( v_k(z) \geq 0 \) in \( B_1 \).
    \item \( |\nabla v_k(z)|^{\gamma} \Delta_p^{\mathrm{N}} v_k(z) = \mathfrak{a}(x_k + d_k z) v_k^q(z) \) in \( B_1 \) in the viscosity sense.
    \item By the local H\"{o}lder regularity of viscosity solutions,
    \[
    v_k(z) \leq \mathrm{C}(n, p, \gamma, \alpha) \cdot d_k^\alpha + \frac{1}{k} \quad \text{for all} \,\,\, z \in \overline{B_{d_k/2}}.
    \]
\end{enumerate}

From the Non-degeneracy result (Theorem \ref{Thm-Non-Deg}) and the previous estimate, we obtain  
\[
0 < \mathrm{C}_{\mathrm{ND}} \cdot \left(\frac{1}{4}\right)^{\beta} \leq \sup_{B_{1/4}} v_k(z) \leq \max\{1, \mathrm{C}(n, p, \gamma, \alpha)\} \cdot \left( d_k^\alpha + \frac{1}{k} \right) = \text{o}(1) \quad \text{as} \quad k \to \infty.
\]
This leads to a contradiction, thereby concluding the proof.
\end{proof}

Next, we establish the porosity of the free boundary.

\begin{proof}[{\bf Proof of Corollary \ref{Corollary-Porosity}}]
Let \( x_0 \) be an arbitrary point in \( \partial\{u > 0\} \), and let \( y_0 \in \partial B_r(x_0) \). For a fixed \( 0 < \tau \ll 1 \), to be determined later, define 
\[
\overline{y} = (1 - \tau)x_0 + \tau y_0,
\]
where \( |\overline{y} - y_0| = (\tau/2)r \). Note that for each \( z \in B_{(\tau/2)r}(\overline{y}) \), we have 
\[
|z - y_0| \leq |z - \overline{y}| + |\overline{y} - y_0| \leq \tau r.
\]
Since \( \overline{y} \) lies on the segment connecting \( x_0 \) and \( y_0 \), it follows that
\begin{align*}
|z - x_0| &\leq |z - \overline{y}| + |x_0 - \overline{y}| = |z - \overline{y}| + \left(|x_0 - y_0| - |\overline{y} - y_0|\right) \\
&\leq \frac{\tau}{2}r + \left(r - \frac{\tau}{2}r\right) \\
&= r.
\end{align*}
By choosing \( \tau \) as in \eqref{tau_constant}, we obtain
\[
B_{(\tau/2)r}(\overline{y}) \subset B_{\tau r}(y_0) \cap B_r(x_0) \subset B_r(x_0) \setminus \partial\{u > 0\}.
\]
Thus, we conclude that \( \partial\{u > 0\} \cap B_{1/2} \) is a \( (\tau/2) \)-porous set. By \cite{Zaj87}, the desired result follows.
\end{proof}

\subsection{Average estimates in the $L^2-$sense:  Proof of Theorem \ref{Thm-Aver-L2}}

In the following, we present the proof of Theorem \ref{Thm-Aver-L2} using an iterative scheme.

\begin{proof}[{\bf Proof of Theorem \ref{Thm-Aver-L2}}]
Without loss of generality, we may assume that \( x_0 = 0 \) by translation and scaling. Define
\[
\mathcal{S}_r[u] := \left( \,\, \intav{{B_1}} \left( |\nabla u(rx)|^\gamma |D^2 u(rx)| \right)^2 dx \right)^{\frac{1}{2(1+\gamma)}}.
\]
It suffices to prove that there exists a universal constant \( \mathrm{C}_* \) such that
\begin{equation}\label{L2-estimates1}
\mathcal{S}_{\frac{1}{2^{k+1}}}[u] \leq \max \left\{ \mathrm{C}_* \left( \frac{1}{2^k} \right)^{\frac{(2+\gamma) m}{(1+\gamma)(1+\gamma-m)}}, \left( \frac{1}{2} \right)^{\frac{(2+\gamma) m}{(1+\gamma)(1+\gamma-m)}} \mathcal{S}_{\frac{1}{2^k}}[u] \right\},
\end{equation}
for all \( k \in \mathbb{N} \) and \( 0 < r \leq \frac{1}{2} \). We proceed by contradiction. Assume that there exist functions \( u_k \) and integers \( j_k \in \mathbb{N} \) such that \eqref{L2-estimates1} does not hold. Consequently,
\begin{align}
\mathcal{S}_{\frac{1}{2^{j_k+1}}}[u_k] &> \max \left\{ k \left( \frac{1}{2^{j_k}} \right)^{\frac{(2+\gamma) m}{(1+\gamma)(1+\gamma-m)}}, \left( \frac{1}{2} \right)^{\frac{(2+\gamma) m}{(1+\gamma)(1+\gamma-m)}} \mathcal{S}_{\frac{1}{2^{j_k}}}[u_k] \right\} \nonumber\\
&\geq \max \left\{ k \left( \frac{1}{2^{j_k}} \right)^{\frac{2+\gamma}{1+\gamma-m}}, \left( \frac{1}{2} \right)^{\frac{(2+\gamma) m}{(1+\gamma)(1+\gamma-m)}} \mathcal{S}_{\frac{1}{2^{j_k}}}[u_k] \right\}.\label{cont-hy}
\end{align}
Now, define the scaled function \( v_k : B_1 \to \mathbb{R} \) by
\[
v_k(x) := \frac{u_k\left( \frac{1}{2^{j_k}} x \right)}{\mathcal{S}_{\frac{1}{2^{j_k+1}}}[u_k]}.
\]
By Theorem \ref{Improved Regularity} and \eqref{cont-hy}, we have
\begin{equation}\label{GSS}
\|v_k\|_{L^\infty(B_1)} \leq \frac{\tilde{\mathrm{C}}}{k}.
\end{equation}
Furthermore, it is straightforward to verify that
\begin{equation}\label{Claudema}
\mathcal{S}_{\frac{1}{2}}[v_k] = 1 \quad \text{and} \quad \mathcal{S}_{1}[v_k] \leq 2^{\frac{(2+\gamma) m}{(1+\gamma)(1+\gamma-m)}}.
\end{equation}
Finally, \( v_k \) satisfies
\begin{equation}
|\nabla v_k|^\gamma \Delta^\mathrm{N}_p v_k (x) = \frac{2^{-j_k(2+\gamma)}}{\mathcal{S}_{\frac{1}{2^{j_k+1}}}^{1+\gamma-m}[u_k]} \mathfrak{a}\left( \frac{1}{2^{j_k}} x \right) v_k^m (x) \defeq \tilde{f}_k,
\end{equation}
in the viscosity sense in \( B_1 \). Moreover, invoking assumption \( \mathrm{A}\ref{Assumption_lambda} \) and \eqref{GSS}, we obtain
\[
\left\| \tilde{f}_k \right\|_{L^\infty(B_1)} \leq \mathrm{C}^m \Lambda_0 \left( \frac{1}{k} \right)^{1+\gamma}.
\]
On the other hand, by applying the \( L^2 \)-bound on the second derivatives from Theorem \ref{Hes-est} and the uniform gradient estimates from Theorem \ref{Thm-grad-est}, we derive
\[
\|\nabla v_k\|_{L^\infty(B_{1/2})} \leq \mathrm{C} \left[ \|v_k\|_{L^\infty(B_1)} + \|\tilde{f}_k\|_{L^\infty(B_1)}^{\frac{1}{1+\gamma}} \right] \leq \frac{\mathrm{C}_*}{k},
\]
and
\[
\left( \,\, \intav{{B_1}} |D^2 v_k (x)|^2 dx \right)^{\frac{1}{2}} < \mathrm{\tilde{C}} < \infty.
\]
Combining these two estimates, we obtain
\[
1 = \mathcal{S}_{\frac{1}{2}} [v_k] \leq \mathrm{\hat{C}}^{\frac{1}{\gamma+1}} \left( \frac{1}{k} \right)^{\frac{\gamma}{1+\gamma}},
\]
which leads to a contradiction for sufficiently large \( k \).
\end{proof}

\section{Liouville-type results: Proof of Theorems \ref{Thm Liouv-I} and \ref{Thm Liouv-II}}
    
In this section, we present the proofs of Liouville-type Theorems \ref{Thm Liouv-I} and \ref{Thm Liouv-II}.

\begin{proof}[{\bf Proof of Theorem \ref{Thm Liouv-I}}]
Without loss of generality, assume that \( x_0 = 0 \). Define the sequence
\[
\varphi_k(x) \coloneqq \frac{u(kx)}{k^{\beta}},
\]
where \( k \in \mathbb{N} \) and \( \beta = (2 + \gamma)/(1 + \gamma - m) \). Since
\[
\nabla \varphi_k(x) = k^{1 - \beta} \nabla u(kx) \quad \text{and} \quad D^2 \varphi_k(x) = k^{2 - \beta} D^2 u(kx),
\]
it follows that
\begin{align*}
|\nabla \varphi_k(x)|^{\gamma} \NDelta \varphi_k(x)
&= |\nabla \varphi_k(x)|^{\gamma} \left( \Delta \varphi_k(x) + (p - 2) \Delta^N_{\infty} \varphi_k(x) \right) \\
&= k^{(1 - \beta)\gamma} |\nabla u(kx)|^{\gamma} \left( k^{2 - \beta} \Delta u(kx) + k^{2 - \beta} (p - 2) \Delta^N_{\infty} u(kx) \right) \\
&= k^{(1 - \beta)\gamma + 2 - \beta} |\nabla u(kx)|^{\gamma} \NDelta u(kx) \\
&= k^{-m \beta} \mathfrak{a}(kx) u^m_+(kx) \\
&= \mathfrak{a}_k(x) (\varphi^m_k)_+(x),
\end{align*}
where \( \mathfrak{a}_k(x) \coloneqq \mathfrak{a}(kx) \). Thus, \( \varphi_k \) is a non-negative viscosity solution to
\[
|\nabla \varphi_k(x)|^{\gamma} \NDelta \varphi_k(x) = \mathfrak{a}_k(x) (\varphi^m_k)_+(x),
\]
and \( \varphi_k(0) = 0 \) for every \( k \in \mathbb{N} \). Let \( r > 0 \) be small, and let \( x_k \in \overline{B}_r \) be such that
\[
\varphi_k(x_k) = \sup_{\overline{B}_r} \varphi_k(x).
\]
Observe that
\begin{equation}\label{o-pequeno2}
\|\varphi_k\|_{L^{\infty}(B_r)} \to 0 \quad \text{as} \quad k \to \infty.
\end{equation}
Indeed, if \( |kx_k| \) is bounded as \( k \to \infty \), then \( |\varphi_k(x_k)| \) is also bounded, say by a constant \( B \), since \( u \) is continuous. Thus,
\[
\varphi_k(x_k) = \frac{u(kx_k)}{k^{\beta}} \leq \frac{B}{k^{\beta}} \to 0 \quad \text{as} \quad k \to \infty.
\]
On the other hand, if \( |kx_k| \to \infty \), then by hypothesis \eqref{o-pequeno}, we have
\[
\varphi_k(x_k) = \frac{|x_k|^{\beta} u(kx_k)}{|kx_k|^{\beta}} \leq r \cdot \frac{u(kx_k)}{|kx_k|^{\beta}} \to 0 \quad \text{as} \quad k \to \infty.
\]
Therefore, by Theorem \ref{Improved Regularity}, we obtain
\[
\varphi_k(x) \leq \text{o}(|x|^{\beta}) \cdot |x|^{\beta}.
\]
Now, suppose there exists a point \( x_0 \in \mathbb{R}^n \) such that \( u(x_0) > 0 \). Since \eqref{o-pequeno2} holds, choose \( k \) sufficiently large such that \( x_0 \in B_{kr} \) and
\[
\sup_{B_r} \varphi_k(x) \leq \frac{u(x_0)}{10k^{\beta}}.
\]
On the other hand, we can estimate
\begin{align*}
\frac{u(x_0)}{k^{\beta}}
&\leq \sup_{B_{kr}} \frac{u(x)}{k^{\beta}} \\
&= \sup_{B_r} \frac{u(kx)}{k^{\beta}} \\
&= \sup_{B_r} \varphi_k(x) \\
&\leq \frac{u(x_0)}{10k^{\beta}},
\end{align*}
which leads to a contradiction. This completes the proof.
\end{proof}

\begin{proof}[{\bf Proof of Theorem \ref{Thm Liouv-II}}]
Fix \( R > 0 \) and consider \( v : \overline{B}_R \to \mathbb{R} \), the viscosity solution to the boundary value problem
\[
\left\{
\begin{array}{rclcl}
|\nabla v(x)|^{\gamma} \NDelta v(x) &=& \mathfrak{a}(x) v_{+}^{m}(x) & \text{in} & B_R, \\
v &=& \sup_{\partial B_R} u(x) & \text{on} & \partial B_R.
\end{array}
\right.
\]
Theorem \ref{existence_uniquiness} ensures that the solution \( v \) is unique. Moreover, by Lemma \ref{Comp-Princ}, we have
\begin{equation}\label{v_great_u}
v \geq u \quad \text{in} \quad B_R.
\end{equation}
From the hypothesis, for \( R \gg 1 \) sufficiently large, it follows that
\begin{equation}\label{theta_CND}
\sup_{\partial B_R} \frac{u(x)}{R^{\beta}} \leq \theta \cdot \mathrm{C}_{\mathrm{ND}},
\end{equation}
for some \( \theta < 1 \). For \( R \gg 1 \), the radial profile \eqref{Radial-profile} implies that the solution \( v = v_R \) is given by
\[
v_R(x) \coloneqq \mathrm{C}_{\mathrm{ND}} \cdot \left[ |x| - R \left( \frac{\sup_{\partial B_R} u(x)}{\mathrm{C}_{\mathrm{ND}}} \right)^{\frac{1}{\beta}} \right]_+^{\beta}.
\]
Combining this with \eqref{v_great_u} and \eqref{theta_CND}, we obtain
\begin{align*}
u(x) &\leq \mathrm{C}_{\mathrm{ND}} \cdot \left[ |x| - R \left( \frac{\sup_{\partial B_R} u(x)}{\mathrm{C}_{\mathrm{ND}}} \right)^{\frac{1}{\beta}} \right]_+^{\beta} \\
&\leq \mathrm{C}_{\mathrm{ND}} \cdot \left[ |x| - R \left( 1 - \theta^{\frac{1}{\beta}} \right) \right]_+^{\beta} \to 0 \quad \text{as} \quad R \to \infty.
\end{align*}
This completes the proof of the theorem.
\end{proof}

\section{The critical equation: Proof of Theorem \ref{Thm-SMP}}

In this section, we address the borderline scenario, i.e., when \( m = \gamma + 1 \). In this context, we prove that a Strong Maximum Principle holds.

\begin{proof}[{\bf Proof of Theorem \ref{Thm-SMP}}]
We proceed by contradiction. Suppose there exists a point \( y \in B_1 \) such that \( u(y) > 0 \). Without loss of generality, assume \( y = 0 \), and define the distance from this point to the zero set of \( u \):
\[
0 < d \coloneqq \dist(0, \{u = 0\}) < \frac{1}{10} \dist(0, \partial B_1).
\]
Define the auxiliary operator
\begin{equation}\label{L-operator}
\mathcal{L}^{1+\gamma}[u] \coloneqq |\nabla u(x)|^{\gamma} \NDelta u(x) - \mathfrak{a}(x) u_+^{1+\gamma}(x) = 0,
\end{equation}
and observe that the function \( v(x) \coloneqq u(0) \chi_{B_1} \) is locally bounded. Indeed, since \( v \) is a viscosity solution to
\[
\mathcal{L}^{1+\gamma}[v] = -\mathfrak{a}(x) u^{1+\gamma}(0) < 0,
\]
and \( u = v \) on \( \partial B_1 \), the Comparison Principle \ref{Comp-Princ} implies \( v(x) \geq u(x) \) in \( B_1 \). Consequently, \( u(0) \geq u(x) \) for all \( x \in B_1 \). 

Next, define an auxiliary function to serve as a barrier:
\begin{equation}\label{beta_0_positive}
\Phi_a(x) \coloneqq
\begin{cases}
e^{-a(d/2)} - \kappa_0, & \text{in } B_{d/2}, \\
e^{-a|x|^2} - \kappa_0, & \text{in } B_{d} \setminus B_{d/2}, \\
0, & \text{in } \mathbb{R}^n \setminus B_{d},
\end{cases}
\end{equation}
where \( \kappa_0 \coloneqq e^{-a d^2} \) for \( a \geq (2/d^2) \). Note that
\begin{equation}
|\nabla \Phi_a(x)| = 2a |x| e^{-a |x|^2} \geq a d e^{-a d^2} \coloneqq \beta_0 > 0 \quad \text{in} \quad B_d \setminus B_{d/2}.
\end{equation}
Moreover, we compute
\begin{align*}
\Delta^{\mathrm{N}}_p \Phi_a(x)
&= 2^{\gamma} a^{\gamma} |x|^{\gamma} e^{-a \gamma |x|^2} \left[ \left( 4a^2 |x|^2 e^{-a |x|^2} - 2a e^{-a |x|^2} \right) + (p - 2) \left( 4a^2 |x|^2 e^{-a |x|^2} - 2a e^{-a |x|^2} \right) \right] \\
&= (2a)^{1+\gamma} |x|^{\gamma} e^{-a (1+\gamma) |x|^2} (p - 1) (2a |x|^2 - 1).
\end{align*}
As a result, for \( x \in B_d \setminus B_{d/2} \), we have
\begin{align*}
\mathcal{L}^{1+\gamma}[\Phi_a(x)]
&= (2a)^{1+\gamma} |x|^{\gamma} e^{-a (1+\gamma) |x|^2} (p - 1) (2a |x|^2 - 1) - \mathfrak{a}(x) \left( e^{-a |x|^2} - \kappa_0 \right)^{1+\gamma} \\
&\geq (2a)^{1+\gamma} \left( \frac{d}{2} \right)^{\gamma} e^{-a (1+\gamma) d^2} (p - 1) \left( 2a \left( \frac{d}{2} \right)^2 - 1 \right) - \|\mathfrak{a}\|_{L^{\infty}(B_1)} \left( e^{-a d^2} - \kappa_0 \right)^{1+\gamma} \\
&\geq 0.
\end{align*}
Now, observe that for any constant \( \theta > 0 \),
\[
\mathcal{L}^{1+\gamma}[\theta \cdot \Phi_a(x)] = \theta^{1+\gamma} \mathcal{L}^{1+\gamma}[\Phi_a(x)] \geq 0 > \mathcal{L}^{1+\gamma}[u] \quad \text{in} \quad B_d \setminus B_{d/2}.
\]
Thus, choosing \( 1 \gg \theta > 0 \) such that
\[
u \geq \theta \cdot \Phi_a \quad \text{on} \quad \partial B_d \cup \partial B_{d/2},
\]
we can apply the Comparison Principle \ref{Comp-Princ} again to obtain
\begin{equation}\label{u_geq_theta}
u \geq \theta \cdot \Phi_a \quad \text{in} \quad B_d \setminus B_{d/2}.
\end{equation}
On the other hand, we can rewrite the equation \eqref{L-operator} as
\[
|\nabla u|^{\gamma} \NDelta u = \mathfrak{a}(x) [u^{\delta}(x)] \cdot u^{1+\gamma - \delta} = \overline{\mathfrak{a}}(x) u^{m}(x),
\]
where \( m \coloneqq 1 + \gamma - \mu \) and the bounded Thiele modulus \( \overline{\mathfrak{a}}(x) \coloneqq \mathfrak{a}(x) [u^{\mu}(x)] \). By Theorem \ref{Improved Regularity} and Remark \ref{remark of improved regularity}, we have
\[
\sup_{B_r(y_0)} u \leq \mathrm{C} r^{2 + \gamma},
\]
for any \( y_0 \in \partial B_d \cap \partial \{u > 0\} \). Now, choose a radius \( r_0 > 0 \) sufficiently small such that
\begin{equation}\label{contradiction_constant}
\mathrm{C} \cdot r_0^{2 + \gamma} \leq \frac{1}{2} \theta \beta_0 \cdot r_0.
\end{equation}
To conclude, combine \eqref{beta_0_positive}, \eqref{u_geq_theta}, and \eqref{contradiction_constant} as follows:
\begin{align*}
0 < \theta r_0 \beta_0
&\leq \theta r_0 \inf_{B_d \setminus B_{d/2}} |\nabla \Phi_a| \\
&\leq \theta r_0 \inf_{B_d \cap \partial B_{r_0}(y_0)} |\nabla \Phi_a| \\
&\leq \theta r_0 \frac{|\Phi_a(x) - \Phi_a(y_0)|}{|x - y_0|} \\
&\leq \sup_{B_d \cap \partial B_{r_0}(y_0)} \theta \cdot |\Phi_a(x) - \Phi_a(y_0)| \\
&\leq \sup_{B_{r_0}(y_0)} \theta \cdot |\Phi_a(x) - \Phi_a(y_0)| \\
&\leq \sup_{B_{r_0}(y_0)} \theta \cdot \Phi_a \\
&\leq \sup_{B_{r_0}(y_0)} u \\
&\leq \mathrm{C} \cdot r_0^{2 + \gamma} \\
&\leq \frac{1}{2} \theta r_0 \beta_0,
\end{align*}
which is a contradiction. This completes the proof.
\end{proof}

\section{Some extensions and final comments}

\subsection*{H\'{e}non-type models with strong absorption}

In this final section, we will discuss H\'{e}non-type equations driven by quasilinear elliptic models in non-divergence form with strong absorption, as follows:

\begin{equation}\label{Eq-Henon-type}
|\nabla u|^\gamma \Delta_p^\mathrm{N} u(x) = \mathfrak{h}(x) u_+^{m}(x) \quad \text{in} \quad B_1,
\end{equation}
where the height function \(\mathfrak{h}: B_1 \to \mathbb{R}_+\) satisfies the following property: given \(\alpha > 0\) and a closed subset \(\mathrm{F} \Subset B_1\), we have
\begin{equation}\label{Homo_cond}
\mathrm{dist}(x, \mathrm{F})^{\alpha} \lesssim \mathfrak{h}(x) \lesssim \mathrm{dist}(x, \mathrm{F})^{\alpha}.
\end{equation}

In this context, we can establish a higher regularity estimate for bounded viscosity solutions along free boundary points.

\begin{theorem}[{\bf Higher regularity estimates}]\label{Hessian_continuity}
Let $u \in \mathrm{C}^0(B_1)$ be a non-negative, bounded viscosity solution of problem \eqref{Eq-Henon-type}. Given $x_0 \in B_{1/2}\cap \mathfrak{h}^{-1}(0)$. If $x_0 \in B_{1/2} \cap \partial \{ u > 0 \}$, then
\begin{equation}\label{Higher Reg}
u(x) \leq \mathrm{C}\cdot \|u\|_{L^{\infty}(B_1)}|x-x_0|^{\frac{\gamma+2+\alpha}{\gamma+1 - m}}
\end{equation}
for every $x \in \{u>0\} \cap B_{1/2}$, where $\mathrm{C} > 0$ depends only on universal parameters.
\end{theorem}

\begin{proof} 
For simplicity and without loss of generality, we assume \( x_0 = \vec{0} \). By combining discrete iterative techniques with continuous reasoning (see \cite{CKS00}), it is well established that proving estimate \eqref{Higher Reg} is equivalent to verifying the existence of a universal constant \( \mathrm{C} > 0 \), such that for all \( j \in \mathbb{N} \), the following holds:
\begin{equation}\label{higher1}
\mathfrak{S}_{j+1} \leq \max\left\{\mathrm{C} 2^{-\hat{\beta}(j+1)}, 2^{-\hat{\beta}} \mathfrak{S}_j\right\},
\end{equation}
where 
\begin{equation}\label{beta_hat}
\mathfrak{S}_j \coloneqq \sup_{B_{2^{-j}}} u(x) \quad \text{and} \quad \hat{\beta} \coloneqq \frac{2 + \gamma + \alpha}{1 + \gamma - m}.
\end{equation}
Suppose, for the sake of contradiction, that \eqref{higher1} fails to hold, \textit{i.e.,} for each \( k \in \mathbb{N} \), there exists \( j_k \in \mathbb{N} \) such that
\begin{equation}\label{higher2}
\mathfrak{S}_{j_k + 1} > \max\left\{k 2^{-\hat{\beta}(j_k + 1)}, 2^{-\hat{\beta}} \mathfrak{S}_{j_k}\right\}.
\end{equation}
Now, for each \( k \in \mathbb{N} \), we define the rescaled function \( v_k: B_1 \rightarrow \mathbb{R} \) as
\[
v_k(x) \coloneqq \frac{u(2^{-j_k} x)}{\mathfrak{S}_{j_k + 1}}. 
\]
Note that
\begin{eqnarray}
0 \leq v_k(x) \leq 2^{\hat{\beta}}; \label{vk1}\\
v_k(0) = 0; \label{vk2}\\ 
\sup_{\overline{B}_{1/2}} v_k(x) = 1. \label{vk3}
\end{eqnarray}
Moreover, observe that
\begin{align}
|\nabla v_k|^\gamma \Delta_p^\mathrm{N} v_k &= \left(\frac{2^{-j_k(2 + \gamma)}}{\mathfrak{S}_{j_k + 1}^{1 + \gamma - m}}\right) \mathfrak{h}\left(2^{-j_k} x\right) (v_k)_+^m(x) \nonumber \\[0.2cm]
&\lesssim \left(\frac{2^{-j_k(2 + \gamma + \alpha)}}{\mathfrak{S}_{j_k + 1}^{1 + \gamma - m}}\right) \mathrm{dist}(x, \mathrm{F})^\alpha (v_k)_+^m \nonumber \\
&\coloneqq \Xi_k \label{vksolution}.
\end{align}
Analogously to Lemma \ref{Flatness Estimate}, we can find \( v_\infty \in \mathrm{C}^{0, \alpha} \) such that \( v_k \to v_{\infty} \) locally uniformly. Moreover, \( v_\infty \) satisfies
\begin{equation}\label{Vasco}
v_\infty(0) = 0 \quad \text{and} \quad \sup_{\overline{B}_{1/2}} v_\infty(x) = 1.
\end{equation}
Now, using \eqref{higher2} and \eqref{vk1}, we obtain 
\[
\Xi_k \leq 2^{\hat{\beta}} \mathrm{diam}(\Omega) \frac{2^{2 + \gamma + \alpha + \hat{\beta} m}}{k} \to 0 \quad \text{as} \quad k \to \infty.
\]
By stability (see \cite[Appendix]{Attou18}), together with \eqref{vksolution}, we conclude that 
\[
|\nabla v_\infty|^\gamma \Delta_p^\mathrm{N} v_\infty = 0 \quad \text{in} \quad B_{8/9}.
\]
Finally, using Theorem \ref{Thm-Equiv-Sol} and V\'{a}zquez's Strong Maximum Principle from \cite{Vazq84}, we deduce that \( v_{\infty} \equiv 0 \) in \( B_{8/9} \), which contradicts \eqref{Vasco}. Thus, \eqref{higher1} holds.

To complete the proof, for each \( r \in (0, 1) \), let \( j \in \mathbb{N} \) be such that
\[
2^{-(j + 1)} \leq r \leq 2^{-j}.
\]
Then, by using \eqref{higher1}, we establish the existence of a universal constant \( \mathrm{C} > 0 \) such that
\[
\sup_{B_r} u \leq \sup_{B_{2^{-j}}} u \leq \frac{\mathrm{C}}{2^{\hat{\beta}}} r^{\hat{\beta}}.
\]
Thus, the theorem is proven.
\end{proof}

\begin{remark}\label{Example6.7}
Another relevant model to consider in \eqref{Problem} is the problem with general distance weights and ``noise terms'', as follows:
\[
|\nabla u(x)|^{\gamma} \Delta_p^\mathrm{N} u(x) = \sum_{i=1}^{l_0} \mathrm{dist}^{\alpha_i}(x, \mathrm{F}_i) u_{+}^{m_i}(x) + g_i(|x|) \quad \text{in} \quad B_1,
\]
where \(\mathrm{F}_i \Subset B_1\) are closed sets, \( m_i \in [0, \gamma + 1) \), \(\alpha_i > 0\), and \( g_i: B_1 \to \mathbb{R} \) are continuous functions such that
\[
\lim_{{x \to \mathrm{F}_i \atop x \notin \mathrm{F}_i}} \frac{g_i(|x|)}{\mathrm{dist}(x, \mathrm{F}_i)^{\kappa_i}} = \mathfrak{L}_i \in [0, \infty) \quad \text{for some} \quad \kappa_i \geq 0.
\]

In this context, non-negative viscosity solutions belong to \( C_{\text{loc}}^{\iota} \), where
\[
\iota \coloneqq \min_{1 \leq i \leq l_0} \left\{\frac{\gamma + 2 + \alpha_i}{\gamma + 1 - m_i}, \frac{\gamma + 2 + \kappa_i}{\gamma + 1}\right\},
\]
along the set \(\displaystyle \left(\bigcap_{i=1}^{l_0} \mathrm{F}_i\right) \cap \partial\{u > 0\} \cap B_{1/2}\).
\end{remark}

\bigskip

\subsection*{Multi-phase problems with strong absorption}

As a final remark, our results can be extended to the class of multi-phase problems, such as those developed in \cite{WYF2025} (see also \cite{BJrDaSRampRic23} and \cite{DaSilvaRic2020} for the fully nonlinear counterpart of such results). In this direction, we consider the model
\[
\left(|\nabla u(x)|^{\gamma} + \sum_{i=1}^{k_0} \mathfrak{a}_i(x) |\nabla u(x)|^{\gamma_i}\right) \Delta_p^\mathrm{N} u(x) = f(x, u) \lesssim \mathfrak{a}(x) u_+^m(x) \quad \text{in} \quad B_1,
\]
where the exponents \(\gamma, \gamma_i \in \mathbb{R}\) satisfy
\[
0 < \gamma \leq \gamma_1 \leq \gamma_2 \leq \cdots \leq \gamma_{k_0} < \infty,
\]
and the modulating functions \(\mathfrak{a}_i: B_1 \to \mathbb{R}\) satisfy
\[
0 < \mathfrak{A}_0 \leq \mathfrak{a}_i \leq \mathfrak{A}_1 \quad \text{for all } 1 \leq i \leq k_0.
\]
Therefore, by following the approach outlined in Theorem \ref{Improved Regularity}, we can prove the following estimate: for \( x_0 \in \partial\{u > 0\} \cap B_{1/2} \),
\[
u(z) \leq \mathrm{C}^{\prime}_{\mathcal{G}}(n, \gamma, \gamma_i, m, p, \mathfrak{A}_1, \lambda_0) \|u\|_{L^{\infty}(B_1)} |z - x_0|^{\frac{\gamma + 2}{\gamma + 1 - m}}
\]
for any point \( z \in \{u > 0\} \cap B_{1/2} \).

Finally, using reasoning similar to \cite[Theorem 1.2]{DaSilvaRic2020}, we can establish the non-degeneracy estimate: for \( y \in \overline{\{u > 0\}} \cap B_{1/2} \), there exists a universal positive constant \(\mathrm{C}^{\prime}_{\mathrm{ND}}\) such that for \( r \in (0, 1/2) \),
\[
\sup_{\overline{B}_r(y)} u \geq \mathrm{C}^{\prime}_{\mathrm{ND}}(n, \gamma, \gamma_i, m, p, \mathfrak{A}_0, \lambda_0) \cdot r^{\frac{\gamma + 2}{\gamma + 1 - m}}.
\]

\subsection*{Acknowledgments}

C. Alcantara was partially supported by the Coordenação de Aperfeiçoamento de Pessoal de Nível Superior (CAPES) – Brasil and by Stone Instituição de Pagamento S.A., under the Project StoneLab. J.V. da Silva has received partial support from CNPq-Brazil under Grant No. 307131/2022-0 and FAEPEX-UNICAMP 2441/23 Editais Especiais - PIND - Projetos Individuais (03/2023). G.S. S\'{a} expresses gratitude for the PDJ-CNPq-Brazil (Postdoctoral Scholarship - 174130/2023-6). This study was partly financed by the CAPES—Brazil - Finance Code 001.

\end{document}